\documentclass{amsart}

\usepackage{amsmath,amsfonts,amssymb}
\usepackage[colorlinks=true]{hyperref}
\hypersetup{urlcolor=blue, citecolor=blue}
\usepackage{geometry}
\usepackage{enumerate}
\usepackage{tikz}
\usepackage{graphicx}
\usepackage{subcaption}

\newtheorem{thm}{Theorem}[section]
\newtheorem{lem}[thm]{Lemma}
\newtheorem{prop}[thm]{Proposition}

\newtheorem{Def}[thm]{Definition}

\newtheorem{rem}[thm]{Remark}
\usepackage{stmaryrd}

\newcounter{cst}
\newcommand{\ctel}[1]{C_{\refstepcounter{cst}\label{#1}\thecst}}
\newcommand{\cter}[1]{C_{\ref{#1}}}

\def\be{\begin{equation}}
\def\ee{\end{equation}}
\def\eps{\epsilon}
\def\x{{\boldsymbol x}}
\def\d{{\rm d}}
\def\ov#1{\overline{#1}}

\def\bv{{\boldsymbol v}}

\def\bw{{\boldsymbol w}}

\def\bLambda{{\boldsymbol{\Lambda}}}
\def\0{{\bf 0}}
\def\1{{\bf 1}}
\def\p{{\partial}}
\def\bbA{{\mathbb A}}
\def\bbB{{\mathbb B}}

\def\bbE{{\mathbb E}}

\def\bbI{{\mathbb I}}

\def\Pp{\mathcal{P}}
\def\dt{{\Delta t}}

\newcommand{\dr}{\partial}
\newcommand{\dx}{\mathrm{d}x}

\def\k{{{K}}}
\def\l{{{L}}}
\newcommand{\ke}{{ {K}^*}}
\def\le{{{L}^*}}
\newcommand{\sig}{ \sigma}
\newcommand{\sige}{\sigma^*}

\newcommand{\M}{\mathfrak{M}}
\newcommand{\Mie}{\mathfrak{M}^*}
\newcommand{\dM}{\dr\M}
\newcommand{\dMie}{\dr\Mie}
\newcommand{\Me}{\Mie\cup\dr\Mie}
\newcommand{\ut}{u_{\T}}

\newcommand{\vt}{v_{\T}}

\newcommand{\Ee}{\mathcal{E}}
\newcommand{\Ds}{{\mathcal{D}}}

\newcommand{\n}{\boldsymbol{{\mathbf{n}}}}
\newcommand{\nksig}{{\mathbf{n}}_{\sig K}}

\newcommand{\nkesige}{\boldsymbol{{\mathbf{n}}}_{\sige\ke}}

\newcommand{\tkele}{\boldsymbol{{\boldsymbol{\tau}}}_{{\ke\!\!\!,\le}}}
\newcommand{\tkl}{\boldsymbol{{\boldsymbol{\tau}}}_{{K\!,L}}}

\newcommand{\petitt}{{\scriptscriptstyle\mathcal{T}}}
\newcommand{\DD}{\mathfrak{D}}
\newcommand{\DDext}{\mathfrak{D}_{ext}}

\newcommand{\DDk}{\mathfrak{D}_{\k}}

\newcommand{\D}{{\scriptstyle\mathcal{D}}}

\newcommand{\Dsig}{{\mathcal{D}_{\sigma,\sigma^*}}}

\newcommand{\petitD}{{\scriptscriptstyle\mathcal{D}}}

\newcommand{\petitk}{{\scriptscriptstyle \mathcal{K}}}

\newcommand{\petitke}{{\scriptscriptstyle \mathcal{K}^*}}

\newcommand{\petitDD}{{\scriptscriptstyle\mathfrak{D}}}

\newcommand{\uk}{u_{\k}}
\newcommand{\ul}{u_{\l}}
\newcommand{\uke}{u_{\ke}}
\newcommand{\ule}{u_{\le}}

\newcommand{\vk}{v_\k}

\newcommand{\vke}{v_\ke}

\newcommand{\xk}{x_{\k}}
\newcommand{\xl}{x_{\l}}

\newcommand{\xke}{x_{\ke}}
\newcommand{\xle}{x_{\le}}
\newcommand{\xD}{x_{\petitD}}

\newcommand{\uM}{u_{\M}}
\newcommand{\udM}{u_{\dM}}
\newcommand{\udMie}{u_{\dMie}}

\newcommand{\uMe}{u_{\Mie}}

\newcommand{\JDD}{J_{\petitDD}}
\newcommand{\JD}{J_{\petitD}}

\newcommand{\msig}{{\rm m}_{\sig}}

\newcommand{\msige}{{\rm m}_{\sige}}

\newcommand{\mk}{{\rm m}_{\k}}
\newcommand{\mke}{{\rm m}_{\ke}}
\newcommand{\md}{{\rm m}_{\Ds}}

\def\O{\Omega}
\newcommand{\dO}{\partial \O}

\newcommand{\dsp}{\displaystyle}

\def\div{\mathrm{div}}
\newcommand{\grad}{\nabla}

\newcommand{\R}{\mathbb R}

\newcommand{\gradDD}{\nabla^{\DD }}

\newcommand{\gradD}{\nabla^{\Ds}}

\newcommand{\divt}{\div^{\T}}

\newcommand{\Pet}{{\mathcal P}^{\T}}

\newcommand{\Pem}{\mathcal{P}^{\M}}
\newcommand{\Pedrm}{\mathcal{P}^{\dr\M}}
\newcommand{\Pemie}{\mathcal{P}^{\Mie}}
\newcommand{\Pedrmie}{\mathcal{P}^{\dr\Mie}}
\newcommand{\Pek}{\mathcal{P}_{K}}
\newcommand{\Peke}{\mathcal{P}_{\ke}}

\newcommand{\divm}{\div^{\M}}
\newcommand{\divdrm}{\div^{\dr\M}}
\newcommand{\divmie}{\div^{\Mie}}
\newcommand{\divdrmie}{\div^{\dr\Mie}}

\newcommand{\divk}{\div_{\k}}
\newcommand{\divke}{\div_{\ke}}

\newcommand{\uh}{u_{h}}
\newcommand{\uhdt}{u_{h,\dt}}
\newcommand{\vh}{v_{h}}

\newcommand{\T}{\mathcal{T}}
\newcommand{\Rt}{\R^\T}
\newcommand{\Ht}{H_\T}
\newcommand{\Htdt}{H_{\T,\dt}}
\newcommand{\HDD}{H_\DD}
\newcommand{\HDDdt}{H_{\DD,\dt}}

\newcommand{\RdDD}{\left(\R^2\right)^{\DD}}

\newcommand{\h}{\mathrm{ size}(\mathcal{T})}
\newcommand{\size}{\mathrm{size}}

\newcommand{\sumdiam}{\ssum_{\Ds\in\DD}}

\newcommand{\sumdiamext}{\ssum_{\Ds\in\DDext}}

\newcommand{\sumpri}{\ssum_{\k\in\M}}
\newcommand{\sumpriext}{\ssum_{\l\in\dr\M}}
\newcommand{\sumdua}{\ssum_{\ke\in{\overline\Mie}}}

\newcommand{\sumdiamsigk}{\mathop{\sum_{\D\in\DDk}}_{\D=\Dsig}}

\DeclareMathOperator*{\ssum}{\sum}

\newcommand{\xib}{{\boldsymbol \xi}}
\newcommand{\phib}{{\boldsymbol \varphi}}

\begin{document}

\title[A nonlinear DDFV scheme for convection diffusion equations]{Numerical analysis of a nonlinear free-energy diminishing Discrete Duality Finite Volume scheme for convection diffusion equations}\thanks{The authors are supported by the Inria teams RAPSODI and COFFEE, the LabEx CEMPI (ANR-11-LABX-0007-01), the GEOPOR project (ANR-13-JS01-0007-01) and the MOONRISE project (ANR-14-CE23-0007).}

\author{Cl\'ement Canc\`es}
\address{
Cl\'ement Canc\`es (\href{mailto:clement.cances@inria.fr}{\tt clement.cances@inria.fr}).
Team RAPSODI, Inria Lille -- Nord Europe, 40 av. Halley, F-59650 Villeneuve d'Ascq, France.
} 

\author{Claire Chainais-Hillairet}
\address{
Claire Chainais-Hillairet (\href{mailto:Claire.Chainais@math.univ-lille1.fr}{\tt Claire.Chainais@math.univ-lille1.fr}).\\
Univ. Lille, CNRS,UMR 8524-Laboratoire Paul Painlev\'e. F-59000 Lille, France.
}

\author{Stella Krell}
\address{
Stella Krell (\href{mailto:stella.krell@unice.fr}{\tt stella.krell@unice.fr}).
Universit\'e de Nice, CNRS, UMR7351-Laboratoire J.-A. Dieudonn\'e. 
F-06100 Nice, France.
}

\begin{abstract}
We propose a nonlinear Discrete Duality Finite Volume scheme to approximate the solutions of drift diffusion equations. 
The scheme is built to preserve  at the discrete level even on severely distorted meshes the energy / energy dissipation relation.
This relation is of paramount importance to capture the long-time behavior of the problem in an accurate way. 
To enforce it, the linear convection diffusion equation is rewritten in a nonlinear form before being discretized. 
We establish the existence of positive solutions to the scheme. Based on compactness arguments,
the convergence of the approximate solution towards a weak solution is established. 
Finally, we provide numerical evidences of the good behavior of the scheme when the discretization parameters 
tend to $0$ and when time goes to infinity.
\end{abstract}

\maketitle

\noindent
{\small {\bf Keywords.}
Convection diffusion equation, Discrete Duality Finite Volumes, convergence analysis, discrete entropy method
\vspace{5pt}

\noindent
{\bf AMS subjects classification. }
65M08, 65M12, 35K20
}

\section{Introduction}

\subsection{Motivation}\label{ssec:goal-pos}

The modeling of systems of interacting particles, like electrons in electronic devices, ions in plasmas, chemical species in biological membranes for instance, leads to systems of evolutive partial differential equations. The knowledge of the large time behavior of such systems is crucial for the understanding of the underlying physical phenomena. In many cases, the relaxation to an equilibrium configuration is based on the second law of thermodynamics and on the dissipation of some entropies. 

Based on works in kinetic theory, mathematicians have intensively developed the entropy method for the study of the large time behavior of different systems of PDEs. Let us mention works on Boltzmann and Landau equations \cite{toscani_villani_2000}, on linear Fokker-Planck equations \cite{carrillo_toscani_1998}, on porous media equations \cite{carrillo_toscani_2000}, on reaction-diffusion systems \cite{desvillettes_fellner_2006, desvillettes_fellner_2015, glitzky_2008}, on drift-diffusion systems for semiconductor devices \cite{gajewski_groger_1986, gajewski_groger_1989, gajewski_gartner_1996}. We also refer to the survey paper \cite{arnold_etal_2004} and to the reference book \cite{Jungel_entropybook}. 
Similar results were obtained based on the interpretation of PDE models as Wasserstein gradient flows~\cite{AGS08}. 
We refer for instance to~\cite{JKO98, BGG12} for linear Fokker Planck equations, 
to \cite{Otto01} for the porous medium equation, to~\cite{BGG13} for granular media. This list is far from being exhaustive.

The knowledge of the large time behavior of such evolution equations, the existence of some entropies which are dissipated along time  are  structural features, as positivity of densities or conservation of mass, that should be preserved at the discrete level by numerical schemes. The question of the large time behavior of numerical schemes has been investigated 
for instance for coagulation-fragmentation models \cite{filbet_2008}, for nonlinear diffusion equations \cite{chainais_jungel_schuchnigg, jungel_schuchnigg}, for reaction-diffusion systems \cite{glitzky_2008, glitzky_2011}, 
for drift-diffusion systems \cite{Chatard2011, chainais_bessemoulin_2016}. These last works show that the Scharfetter-Gummel numerical fluxes, first introduced in \cite{SG69} for the approximation of convection-diffusion fluxes and widely used later for the simulation of semiconductor devices, preserve the thermal equilibrium. Their use in the numerical approximation of drift-diffusion systems ensure the exponential decay towards equilibrium of the numerical scheme. 
Unfortunately, such numerical fluxes can only be applied in two-points flux approximation finite volume schemes and therefore  on restricted meshes. Moreover, they do not extend to anisotropic convection-diffusion equations. 

Therefore, it seems crucial to propose new finite volume schemes which preserve the large-time behavior of anisotropic convection-diffusion equations and which apply on almost general meshes. In \cite{CG2016}, the authors proposed and analyzed a VAG scheme satisfying these prescribed properties. 
In this work, we propose and study the convergence analysis of a nonlinear free-energy diminishing discrete duality finite volume scheme~\cite{CCK_FVCA}. 

\subsection{Presentation of the continuous problem}

We focus on a very basic drift-diffusion equation with potential convection and anisotropy. 
Let $\Omega$ be a polygonal connected open bounded subset of $\R^2$ and let $T>0$ be a finite time horizon. The problem writes:\begin{subequations}\label{pb}
\begin{align}
&\partial_t u+ \div {\mathbf J} =0, \ \mbox{ in } Q_T = \Omega\times (0,T),\label{cons}\\
& {\mathbf J}=-\bLambda\grad u -u \bLambda \grad V,\ \mbox{ in } Q_T,\label{flux}\\
&{\mathbf J}\cdot {\mathbf n} =0,\ \mbox{ on } \p \Omega\times (0,T),\label{CLneum}\\
&u(\cdot,0)=u_0,\ \mbox{ in }\Omega,\label{CI}
\end{align}
\end{subequations}
with ${\mathbf n} $ the outward unit normal to $ \p \Omega$ and the following assumptions on the data:
\begin{enumerate}\label{hyp}
\item[(A1)] The initial data $u_0$ is measurable, nonnegative and satisfies
\be\label{eq:A1}
\int_\Omega u_0 \d\x >0\quad  \mbox{ and }\quad  \int_\O H(u_0)\d\x < \infty, 
\ee
where $ H(s)=s\log s -s+1$ for all $s\ge 0$.
\item[(A2)] The exterior potential $V$ belongs to $C^1({\overline\Omega},\R)$. Without loss of generality, we assume that $V \ge 0$ in $\ov \O$.
\item[(A3)] The anisotropy tensor $\bLambda$ is supposed to be bounded 
(i.e., $\bLambda \in L^\infty(\O)^{2\times 2}$), symmetric (i.e., $\bLambda= \bLambda^T$ a.e. in $\O$), and 
uniformly elliptic: there exist $\lambda_m>0$ and $\lambda^M>0$ such that 
\be\label{eq:elliptic}
\lambda_m |\bv|^2 \leq \bLambda(\x) \bv \cdot \bv \leq \lambda^M |\bv|^2, \qquad \text{for all $\bv \in \R^2$ and almost all $\x \in \O$}.
\ee
\end{enumerate}

The flux ${\mathbf J}$ can be reformulated in the nonlinear form
$$
{\mathbf J} = -u \bLambda \grad (\log u+V).
$$
Testing equation~\eqref{cons}  by $\log(u)+V$ leads to the so-called {\em energy/energy dissipation} relation (energy/dissipation for short)
\be\label{eq:EI}
\dsp \frac{d {\mathbb E}}{dt}+{\mathbb I}=0, 
\ee
where the free energy $\mathbb E$ and the dissipation $\mathbb I$ for \eqref{pb} are respectively defined by
 \begin{align}
{\mathbb E}(t)=& \int_{\Omega} (H(u)+Vu)(\x,t)dx, \label{eq:E} \\
 {\mathbb I}(t)= & \int_\Omega u\bLambda \grad (\log u+V) \cdot \grad (\log u+V) dx. \label{eq:I}
 \end{align}
Since $u$ is nonnegative, so does ${\mathbb I}$ and the free energy $\mathbb E$ is decaying with time.
As highlighted for instance in~\cite{AMTU01} and \cite{BGG12},
the solution $u$ to~\eqref{pb} converges towards the steady-state
$$u_\infty=\left(\int_\O u_0 dx / \int_\O e^{-V} dx\right) e^{-V}$$ 
when time goes to infinity. In the case where $\bLambda$ does not depend on $\x$ and where both $\O$ and $V$ 
are convex, this convergence is exponentially fast. 

The energy/energy dissipation relation~\eqref{eq:EI} provides a control on the Fisher information 
\be\label{eq:Fisher-chainrule}
\iint_{Q_T} u |\grad \log(u)|^2 \d\x \d t 
= 4 \iint_{Q_T} |\grad \sqrt{u}|^2 \d\x \d t \leq C. 
\ee
Thus it is natural to seek the solution in the space
$$
\left\{ u: Q_T \to \R_+\; \middle| \; \int_\O H(u(\x,\cdot)) \d\x \in L^\infty(0,T) \;\text{and}\; 
\sqrt{u} \in L^2(0,T;H^1(\O))\right\}.
$$
This motivates the following notion of weak solution. 
\begin{Def}\label{Def:weak}
A function $u: Q_T \to \R_+$ is said to be a weak solution to the problem  \eqref{pb} if $H(u) \in L^\infty(0,T,L^1(\O))$, $\sqrt u \in L^2(0,T,H^1(\O))$, and~\eqref{pb}
is satisfied in the distributional sense, i.e., for all 
$\varphi \in C^\infty_c(\ov\O\times[0,T))$, there holds
\be\label{eq:weak}
\iint_{Q_T} u \p_t \varphi \, \d\x \d t + \int_\O u_0 \varphi(\cdot,0) \d\x 
- \iint_{Q_T} \left(  u \grad V + \grad u  \right) \cdot \bLambda \grad \varphi \, \d\x \d t = 0.
\ee
\end{Def}


\subsection{Outline of the paper}

In Section \ref{sec:scheme}, we introduce the numerical scheme and state the main results of the paper: existence of a positive solution to the scheme and convergence of a sequence of approximate solutions towards a weak solution. The existence of a solution to the scheme is established in Section \ref{sec:existence}. It strongly relies on the conservation of mass at the discrete level and on a discrete counterpart of an energy/dissipation estimate. Section \ref{sec:convergence} is devoted to the proof of convergence of the scheme. The effective behavior of the numerical method is eventually discussed  in Section~\ref{sec:num}. 
It is shown that the method is second order accurate w.r.t. space in $L^2$ norm, whereas the approximate gradient 
super-converges with observed order 3/2. Moreover, the method exhibit a very accurate long-time behavior.

\section{Presentation of the scheme and main results}\label{sec:scheme}

\subsection{Meshes and notations}\label{sec-meshes}


In order to define a DDFV scheme, as for instance in \cite{DO_2005,ABH_2007}, we need to introduce three different meshes -- the primal mesh, 
the dual mesh and the diamond mesh -- and some associated notations.

The primal mesh denoted $\overline{\M}$ is composed of the interior primal mesh $\M$ (a partition of $\Omega$ with polygonal control volumes) and the set $\dr\M$ of boundary edges seen as degenerate control volumes. 
For all $K\in \overline{\M}$, we define $x_K$ the center of $K$. The family of centers is denoted by ${\mathfrak X}=\{\xk,\k\in \overline{\M}\}$.

Let ${\mathfrak X}^*$ denote the set of the vertices of the primal control volumes in $\overline{\M}$. 
Distinguishing the interior vertices from the vertices lying on the boundary, 
 we split ${\mathfrak X}^*$ into ${\mathfrak X}^*={\mathfrak X}_{int}^*\cup {\mathfrak X}_{ext}^*$. To any point $\xke\in {\mathfrak X}_{int}^*$, we associate the polygon $\ke$, whose vertices are
 $\{\xk\in {\mathfrak X}/\xke\in{\overline\k},\k\in\M\}$. 
The set of these polygons defines the interior dual mesh denoted by $\M^*$. To any point $\xke\in {\mathfrak X}_{ext}^*$, we then associate the polygon $\ke$,
 whose vertices are $\{\xke\}\cup\{\xk\in X/\xke\in{\bar\k},\k\in\overline{\M}\}$. The set of these polygons is denoted by $\dr\M^*$ 
called the boundary dual mesh and the dual mesh is $\Me$, denoted by $\overline{\Mie}$.

For all neighboring primal cells $\k$ and $\l$, we assume that $\dr\k\cap\dr\l$ is a segment, corresponding to an edge of the mesh $\M$, 
denoted by $\sigma=\k\vert\l$. Let $\Ee$ be the set of such edges. We similarly define the set $\Ee^*$ of the edges  of the dual mesh.
For each couple $(\sigma,\sigma^*)\in\Ee\times\Ee^*$ such that $\sigma=\k\vert\l=(\xke,\xle)$ and $\sigma^*=\ke\vert\le=(\xk,\xl)$,
 we define the quadrilateral diamond cell $\Dsig$ whose diagonals are $\sigma$ and $\sigma^*$, as shown on Figure \ref{fig_diamonds}. If $\sigma\in \Ee\cap\dr\Omega$,
 we note that the diamond degenerates into a triangle. The set of the diamond cells defines the diamond mesh $\DD$.
 It is a partition of $\Omega$.
We can rewrite $\DD=\DD^{ext}\cup \DD^{int}$ where $\DD^{ext}$ is the set of all the boundary diamonds and $\DD^{int}$ the set of all the interior diamonds.

\begin{figure}[htb]
\begin{center}
\begin{tikzpicture}[scale=1.]
\node[rectangle,scale=0.8,fill=black!50] (xle) at (0,0) {};
\node[circle,draw,scale=0.5,fill=black!5] (xl) at (2,1.3) {};
\node[rectangle,scale=0.8,fill=black!50]  (xke) at (0,4) {};
\node[circle,draw,scale=0.5,fill=black!5] (xk) at (-2,2.3) {};
\draw[line width=1pt] (xle)--(xke);
\draw[dashed, line width=1pt] (xk)--(xl);
\draw[dash pattern=on 2pt off 3pt on 6pt off 3pt,line width=2pt] (xk)--(xke)--(xl)--(xle)--(xk);

\node[yshift=-8pt] at (xle){$\xle$};
\node[yshift=8pt] at (xke){$\xke$};
\node[xshift=10pt] at (xl){$\xl$};
\node[xshift=-10pt] at (xk){$\xk$};

\draw[->,line width=1pt] (-2+3.*0.4,2.3-3.*0.1)--(-2+4.8*0.4,2.3-4.8*0.1);
\draw[->,line width=1pt] (-2+3.*0.4,2.3-3.*0.1)--(-2+3.*0.4-1.8*0.1,2.3-3.*0.1-1.8*0.4);
\draw[->,line width=1pt] (0,2.7)--(0,2.);
\draw[->,line width=1pt] (0,2.7)--(.7,2.7);
\node[right] at (0,2.3){$\tkele$};
\node[right] at (0.,2.95){$\nksig$};
\node at (-0.5,2.2){$\tkl$};
\node at (-0.45,1.25){$\nkesige$};

\draw[line width=1pt]  (2.2,4)--(2.8,4);
\node[right,xshift=8] at (2.8,4){$\sigma=\k\vert\l$, edge of the primal mesh};
\draw[dashed, line width=1pt] (2.2,3.5)--(2.8,3.5);
\node[right,xshift=8] at (2.8,3.5){$\sigma^*=\ke\vert\le$, edge of the dual mesh};
\draw[dash pattern=on 2pt off 3pt on 6pt off 3pt,line width=2pt] (2.2,3)--(2.8,3);
\node[right,xshift=8] at (2.8,3.){Diamond $\Dsig$};
\node[rectangle,scale=0.8,fill=black!50] at (2.8,2.5) {};
\node[circle,draw,scale=0.5,fill=black!5] at (2.8,2.) {};
\node[right,xshift=8] at (2.8,2.5){Vertices of the primal mesh};
\node[right,xshift=8] at (2.8,2.){Centers of the primal mesh};
\node [right]at (0,1.5) {$x_\Ds$};
\node at (0,1.8) {$\bullet$};

\node[rectangle,scale=0.8,fill=black!50] (xle2) at (11,0) {};
\node[circle,draw,scale=0.5,fill=black!5] (xl2) at (11,2) {};
\node[rectangle,scale=0.8,fill=black!50]  (xke2) at (11,4) {};
\node[circle,draw,scale=0.5,fill=black!5] (xk2) at (9.5,2.3) {};
\draw[line width=1pt] (xle2)--(xke2);
\draw[dashed, line width=1pt] (xk2)--(xl2);
\draw[dash pattern=on 2pt off 3pt on 6pt off 3pt,line width=2pt] (xle2)--(xk2)--(xke2);

\node[yshift=-8] at (xle2){$\xle$};
\node[yshift=8] at (xke2){$\xke$};
\node[xshift=10] at (xl2){$\xl$};
\node[xshift=-10] at (xk2){$\xk$};

\end{tikzpicture}
\end{center}
\caption{Definition of the diamonds $\Dsig$ and related notations.}\label{fig_diamonds}
\end{figure}
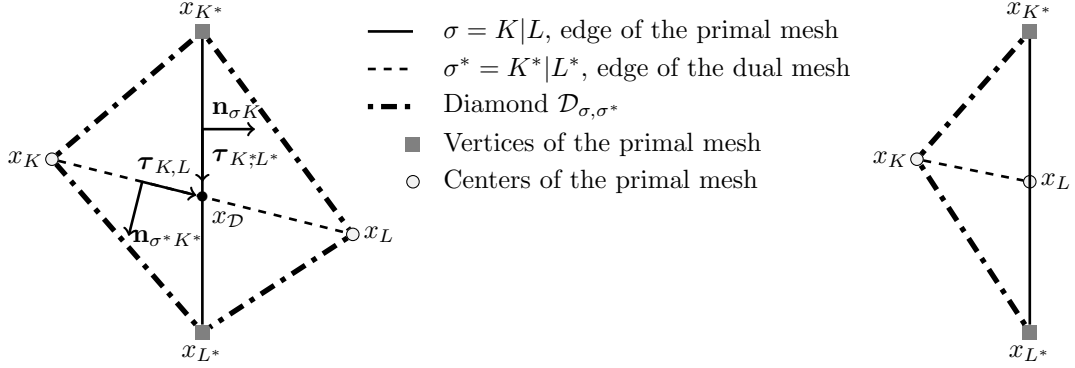

Finally, the DDFV mesh is made of  $\T=(\overline{\M},\overline{\Mie})$ and $\DD$.  For each primal or dual cell $M$ ($M\in {\overline \M} $ or 
$M\in {\overline \Mie}$), we define ${\rm m}_M$ the measure of $M$, $\Ee_M$ the set of the edges of $M$ 
(it coincides with the edge $\sigma=M$ if $M\in\dr\M$), $\DD_M$ the set of diamonds $\Dsig\in\DD$ such that ${\rm m}(\Dsig\cap M)>0$, 
and $d_M$ the diameter of $M$. 

For a diamond $\Dsig$, whose vertices are $(\xk,\xke,\xl,\xle)$, 
we define:
$x_\Ds$ the center of the diamond cell $\Ds$: $\{\xD\}=\sig\cap\sige$, $\msig$ the length of the primal edge $\sig$,
$\msige$ the length of the dual edge $\sige$, $d_\Ds$ the diameter of $\Ds$, $\alpha_\Ds$ the angle between $(\xk,\xl)$ and $(\xke,\xle)$. 
We will also use two direct basis $(\tkele,\nksig)$ and $(\nkesige,\tkl)$, where $\nksig$ is the unit normal to $\sigma$, outward $\k$, $\nkesige$ is the unit normal to $\sigma^*$, outward $\ke$, $\tkele$ is the unit tangent vector to $\sigma$, oriented from $\ke$ to $\le$, $\tkl$ is the unit tangent vector to $\sigma^*$, oriented from $\k$ to $\l$. 
Denoting by $\md$ the $2$-dimensional Lebesgue measure of $\Ds$, one has 
\be\label{eq:m_Dd}
\md = \frac12 \msig\msige \sin(\alpha_\Ds), \qquad \forall \Ds = \Ds_{\sig, \sig^\ast} \in \DD.
\ee

We define two local regularity factors $\theta_\Ds, \tilde \theta_\Ds$ of the diamond cell $\Ds= \Ds_{\sig, \sig^\ast} \in \DD$ by 
\be\label{eq:thetaD}
\theta_\Ds = \frac1{2\sin(\alpha_\Ds)}\left(\frac{\msig}{\msige} + \frac{\msige}{\msig}\right) \ge 1, 
\qquad 
\tilde \theta_\Ds = \max \left( \max_{K \in \M_\Ds} \frac{\md}{{\rm m}_{\Ds \cap K}}\, ; \,\max_{K^* \in \M^*_\Ds} \frac{\md}{{\rm m}_{\Ds \cap K^*}} \right).
\ee
In what follows, we assume that there exists $\theta^\star \ge 1$ such that 
\be\label{eq:theta_star}
1 \leq \theta_\Ds, \tilde \theta_\Ds \leq \theta^\star, \qquad \forall \Ds \in \DD.
\ee
In particular, this implies that 
\be\label{eq:alpha-star}
\sin(\alpha_\Ds) \ge \frac1{\theta^\star}, \qquad \forall \Ds \in \DD.
\ee
Moreover, owing to the definition of $\tilde \theta_\Ds$ and to~\eqref{eq:theta_star}, one has 
\be\label{eq:reg-meas}
\sum_{\Ds \in \DD_K} \md \leq \theta^\star \mk
\quad \text{and}\quad
\sum_{\Ds \in \DD_{K^*}} \md \leq \theta^\star \mke
\ee
Finally, we define the size of the mesh: 
${\rm size}(\T)= {\max_{\Ds\in\DD}} {\rm d}_{\Ds}.$

\subsection{Discrete unknowns and discrete operators}\label{ssec:operators}

We first define the different sets of discrete unknowns. As it is usual for DDFV methods,
we need several types of degrees of freedom to represent scalar and  vector  fields in the discrete setting.
We  introduce  $\Rt$  the linear space of scalar fields constant on the cells of ${\overline \M}$ and ${\overline \Mie}$:
\[
 \ut\in\Rt\Longleftrightarrow  \ut =\left(\left(\uk\right)_{\petitk\in{\overline \M}},\left(\uke\right)_{\petitke\in{\overline \Mie}}\right)
\]
and  $\RdDD$  the linear space of vector fields constant on the diamonds: 
$$ 
\xib_\DD \in \RdDD\Longleftrightarrow \xib_\DD=\left(\xib_\Ds\right)_{\Ds\in\DD}.
$$
Let us mention that we similarly denote by $\R^\DD$ the set of scalar fields constant on the diamonds.

Then,we define the positive semi-definite bilinear form\footnote{Although it mimics the continuous $L^2(\O)$ scalar product, the bilinear form 
$\llbracket\cdot,\cdot\rrbracket_\T $ is not a scalar product since it does not involve 
the primal boundary edges $\p\M$. It is therefore not definite.} $\llbracket\cdot,\cdot\rrbracket_\T $ 
on $\Rt $ and the scalar product $(\cdot,\cdot)_{\bLambda, \DD} $ on $\RdDD$ by 
$$
\begin{aligned}
&\dsp\left\llbracket\vt,\ut\right\rrbracket_\T&&=&&\frac{1}{2}\left(\sumpri\mk\uk \vk+\sumdua \mke\uke\vke
\right), \ \ \ \forall \ut,\vt\in \Rt,\\
&\dsp\left(\xib_\DD,\phib_\DD\right)_{\bLambda, \DD}&&=&&\sumdiam\md\ 
 \xib_\Ds\cdot \bLambda^\Ds \phib_\Ds,\ \ \ \forall \xib_\DD,\phib_\DD\in\RdDD,
\end{aligned}
$$
where 
$$\bLambda^\Ds = \frac1{\md} \int_{\Ds} \bLambda(\x) \d\x, \qquad \forall \Ds \in \DD.$$
We denote by $\|\cdot\|_{\bLambda,\DD}$ the Euclidian norm associated to the scalar product $\left(\cdot, \cdot\right)_{\bLambda,\DD}$, i.e.,
$$
\left\| \xib_\DD \right\|_{\bLambda,\DD}^2 = \left(\xib_\DD, \xib_\DD \right)_{\bLambda,\DD}, \qquad \forall \xib_\DD \in \left(\R^2\right)^\DD.
$$

The DDFV method is based on the definitions of a discrete gradient and  of a discrete divergence, which are linked by duality formula as shown in \cite{DO_2005}.
The discrete gradient has been introduced in \cite{CVV99} and developed in \cite{DO_2005}.
It  is a mapping from $\Rt$ to $\RdDD$ defined 
by  $\gradDD \ut =\dsp\left(\gradD \ut\right)_{\Ds\in\DD}$ for all  $\ut\in \Rt$, 
where
\[
 \gradD \ut =\frac{1}{\sin(\alpha_\Ds)}\left(\frac{\ul-\uk}{\msige}\nksig+\frac{\ule-\uke}{\msig}\nkesige\right), \quad \forall  \Ds\in\DD.
\]
Using \eqref{eq:m_Dd}, the discrete gradient can be equivalently written:
\[
 \gradD \ut =\frac{1}{2\md}\left(\msig(\ul-\uk)\nksig+\msige(\ule-\uke)\nkesige\right), \quad \forall  \Ds\in\DD.
\]

The discrete divergence has been introduced in \cite{DO_2005}. It is a mapping $\divt$ from $\RdDD$ to $\Rt$ defined for all 
 $\xib_\DD\in \RdDD$ by 
\[
 \divt \xib_\DD=\left(\divm \xib_\DD,\divdrm \xib_\DD,\divmie \xib_\DD,\divdrmie \xib_\DD \right),
\]
  with $\divm \xib_\DD=\left( \divk \xib_\DD\right)_{\petitk\in\M}$, $\divdrm \xib_\DD=0$,
 $\divmie\xib_\DD=\left( \divke \xib_\DD\right)_{\petitke\in\Mie} $
 and $\divdrmie\xib_\DD=\left( \divke \xib_\DD\right)_{\petitke\in\dr\Mie} $  such that:
$$
\forall\ \k\in\M,\divk \xib_\DD=\dsp\frac{1}{\mk}\sumdiamsigk\msig\ \xib_{\Ds}\cdot\nksig,
$$
and analogous definitions for $\divke\xib_\DD$ for $\ke\in\overline\Mie$.

In \cite{ABK2010}, the authors study the convergence of DDFV schemes for degenerate hyperbolic-parabolic problems. They show that a penalization operator is needed in order to establish the convergence proof. Indeed, this penalization operator ensures that the two components of a discrete function (reconstructions on the primal and dual meshes) converge to the same limit. For similar reasons (see Section \ref{sec:convergence}), we consider the same penalization operator ${\mathcal P}^\T:\Rt\to\Rt$ as in \cite{ABK2010}
and \cite{CKM15}. It is defined for all $\ut\in\Rt$ by 
\[
 \Pet \ut=\left(\Pem \ut,\Pedrm \ut,\Pemie \ut,\Pedrmie \ut \right),
\]
  with $\Pem \ut=\left( \Pek \ut\right)_{K\in\M}$, $\Pedrm \ut=0$,
 $\Pemie\ut=\left( \Peke \ut\right)_{\ke\in\Mie} $
 and $\Pedrmie\ut=\left( \Peke \ut\right)_{\ke\in\dr\Mie} $ such that, for a given parameter $\beta\in (0,2)$,
\begin{center}
$
\begin{aligned}
\forall\ \k\in\M,&\ \ \Pek \ut=\dsp\frac{1}{\mk}\frac{1}{\h^\beta}\sumdua {\rm m}_{K\cap\ke} (\uk-\uke),\\
\forall\ \ke\in{\overline \Mie},&\ \ \Peke \ut=\dsp\frac{1}{\mke}\frac{1}{\h^\beta}\sumpri {\rm m}_{K\cap\ke} (\uke-\uk).
\end{aligned}
$\end{center}

It clearly satisfies : for all $\ut,\vt\in\Rt$,
\be 
\llbracket \Pet \ut,\vt \rrbracket_\T=\frac{1}{2}\frac{1}{{\rm Size}(\T)^\beta}
\sumdua\sumpri{\rm m}_{K\cap\ke}(u_K - u_\ke)(v_K-v_\ke)\quad  \mbox{ with } \beta\in (0,2).
\label{prop_pen}
\ee

Finally, we introduce a reconstruction operator on diamonds $r^\DD$. It is a mapping from  $\R^\T$ to  $\R^\DD$ defined  for all  $\ut\in \Rt$ by  $r^\DD[\ut] =\left(r^\Ds(\ut)\right)_{\Ds\in\DD}$, 
where for $\Ds\in\DD$, whose vertices are $x_K$, $x_L$, $x_\ke$, $x_\le$,
\be\label{eq:rD}
r^\Ds(\ut)=\frac{1}{4}(\uk+\ul+\uke+\ule).
\ee

We conclude this section with a remark on the particular structure of the scalar product of two discrete gradients $(\gradDD \ut,\gradDD \vt)_{\bLambda, \DD}$ for $\ut,\vt\in \Rt$. Indeed, for  $\ut\in\Rt$ and $\Ds \in \DD$, we  define $\delta^\Ds \ut$ by 
$$
\delta^\Ds \ut=\left(\begin{array}{c} u_K-u_L\\ \uke-\ule\end{array}\right).
$$
Then, we can write 
$$
(\gradDD \ut,\gradDD \vt)_{\bLambda, \DD}=\sumdiam \delta^\Ds \ut\cdot {\mathbb A}^{\Ds}\delta^\Ds \vt,
$$
where the local matrices ${\mathbb A}^{\Ds}$ are defined by
\be\label{eq:AD}
\bbA^\Ds = \frac{1}{4 \md} \begin{pmatrix}
\msig^2 (\bLambda^\Ds \n_{\k,\sig}\cdot  \n_{\k,\sig})&  \msig {\msige}  (\bLambda^\Ds \n_{\k,\sig}\cdot  \n_{\ke,\sige}) \\
\msig \msige  (\bLambda^\Ds \n_{\k,\sig}\cdot  \n_{\ke,\sige})& \msige^2  (\bLambda^\Ds \n_{\ke,\sige}\cdot  \n_{\ke,\sige})
\end{pmatrix} = 
\begin{pmatrix}
A^\Ds_{\sig,\sig} & A^\Ds_{\sig,\sige} \\
A^\Ds_{\sig,\sige} & A^\Ds_{\sige,\sige} 
\end{pmatrix}.
\ee
It follows from elementary calculations left to the reader that 
the condition number of $\bbA^\Ds$ with respect to the 2-norm can be bounded by 
\be\label{eq:condAD}
{\rm Cond}_2(\bbA^\Ds) \leq {\rm Cond}_2(\bLambda^\Ds)  \left( \theta_\Ds + \sqrt{\theta_\Ds^2 - \frac{1}{{\rm Cond}_2(\bLambda^\Ds) }}\right)^2 < 4 (\theta^\star)^2 \frac{\lambda^M}{\lambda_m}, \qquad \forall \Ds \in \DD.
\ee

\subsection{The nonlinear DDFV scheme}\label{ssec:scheme}

Let $N_T$ be a positive integer, we consider for simplicity the constant time step is given by $\Delta t=T/{N_T}$.
For $n \in \{0,\dots, N_T\}$, we denote by $t^n = n \Delta t$.
We first discretize the initial condition by taking the mean values of $u_0$, i.e., 
\be\label{eq:u0K}
u_K^0 = \frac1{m_K} \int_K u_0 \d\x, \quad u_{K^\ast}^0 = \frac1{m_{K^\ast}} \int_K u_0 \d\x, \qquad \forall K \in \M, \; \forall K^\ast \in \ov{\M^\ast},\quad \udM^0=0,
\ee
 and the exterior potential $V$ by taking its nodal values on the primal and dual cells, i.e.,
 \be\label{eq:VK}
 V_K = V(\x_K), \quad V_{K^\ast} = V(\x_{K^\ast}), \qquad  \qquad \forall K \in \ov \M, \; \forall K^\ast \in \ov{\M^\ast}.
 \ee
 It defines in particular $\ut^0$ and $V_\T$. 

The scheme requires a stabilization parameter denoted by $\kappa>0$. It is a fixed parameter.
Then, for all $n\geq 0$, we look for $\ut^{n+1}\in (\R_+^\ast)^\T$ solution to the following variational formulation:
\begin{subequations}\label{scheme}
\begin{align}
&\Bigl\llbracket\dsp\frac{\ut^{n+1}-\ut^n}{\Delta t}, \psi_\T\Bigl\rrbracket_\T+T_{\DD}(\ut^{n+1}; g_\T^{n+1},\psi_\T)+
\kappa \left\llbracket {\mathcal P}^\T g_\T^{n+1},\psi_\T\right\rrbracket_\T =0,\quad \forall \psi_\T\in\Rt,\label{sch_formcompacte}\\
& T_{\DD}(\ut^{n+1}; g_\T^{n+1},\psi_\T)=\sumdiam r^\Ds(\ut^{n+1}) \, \delta^\Ds g_\T^{n+1}\cdot {\mathbb A}^{\Ds}\delta^\Ds \psi_\T,\label{sch_defTd}\\
& g_\T^{n+1}=\log (\ut^{n+1})+V_\T.\label{sch_defgt}
\end{align}
\end{subequations}
Let us mention that, in view of its implementation, the scheme can be rewritten on each mesh as follows:
\begin{subequations}\label{scheme_mesh}
\begin{align}
 &\frac{\uM^{n+1}-\uM^{n}}{\dt}+\divm(\JDD^{n+1})+\kappa \Pem g_\T^{n+1}=0,\label{eq:sch_primal}\\
 &\frac{\uMe^{n+1}-\uMe^{n}}{\dt}+\divmie(\JDD^{n+1})+\kappa \Pemie g_\T^{n+1}=0,\label{eq:sch_dual}\\
 & \frac{\udMie^{n+1}-\udMie^{n}}{\dt}+\divdrmie(\JDD^{n+1})+\kappa \Pedrmie g_\T^{n+1} =0,\label{eq:sch_dualbord}\\
 &\JDD^{n+1}=-r^\DD[\ut^{n+1}]\bLambda^\DD \gradDD g_\T^{n+1},\label{eq:sch_diam}\\
 & \msig\JD^{n+1}\cdot{\mathbf n}=0,\qquad\forall\ \Ds=\Dsig\in\DDext\label{eq:sch_primalbord}.
\end{align}
\end{subequations}

\subsection{Functional spaces}\label{ssec:functspaces}

For a given vector $\ut$ defined on a DDFV mesh $\T$ of size $h$, one usually reconstructs three different approximate solutions : $u_{h,\M}$ is a piecewise constant reconstruction on the primal mesh,  $u_{h,\overline{\Mie}}$ is a piecewise constant reconstruction on the dual mesh and $u_h$ is the mean value of $u_{h,\M}$ and $u_{h,\overline{\Mie}}$. They are defined by
$$
u_{h,\M}= \sumpri\uk\boldsymbol{1}_\k, \quad u_{h,\overline{\Mie}}=\sumdua\uke\boldsymbol{1}_\ke \mbox{ and }
u_h=\frac{1}{2}(u_{h,\M}+u_{h,\overline{\Mie}}).
$$
Then, the set of the approximate solutions is denoted by $\Ht$:
\begin{multline}\label{defHt}
 \Ht=\Bigg\{\uh\in L^1(\O)\ /\ \exists
 \ut=\left(\left(\uk\right)_{\petitk\in{\overline \M}},\left(\uke\right)_{\petitke\in{\overline \Mie}}\right) \in\Rt\\
\mbox{ such that } \uh=\frac12\sumpri\uk\boldsymbol{1}_\k+\frac12\sumdua\uke\boldsymbol{1}_\ke\Bigg\}.
\end{multline}
In the sequel, we will also need some reconstruction of the approximate solutions on the diamond cells. Thanks to the reconstruction operator on diamonds $r^\DD$, we can define $u_\Ds= r^\Ds(\ut)$ for all $\Ds\in\DD$ for instance. Therefore, we can define a  piecewise constant function on diamond cells $u_{h,\DD}$ by $u_{h,\DD}=\sumdiam u_\Ds\boldsymbol{1}_\Ds$. The set of such functions is denoted $\HDD$.

 For a function $\uh\in\Ht$, we define its approximate gradient $\nabla^h\uh\in (\HDD)^2$ by
 $$
 \nabla^h\uh=\dsp\sumdiam\gradD\ut\boldsymbol{1}_\Ds.
 $$
 
 As the problem  \eqref{pb} is an evolutive problem, the numerical scheme \eqref{scheme} defines $\ut^{n}\in\Rt$ for all $n\in \{0,\ldots,N_T\}$. We consider approximate solutions which are piecewise constant in time. 
 Therefore, we define 
 the space-time approximation spaces $\Htdt$ and $\HDDdt$ based respectively
 on $\Ht$ and $\HDD$:
  \begin{align*}
 \Htdt&=\Big\{\uhdt\in L^1(Q_T)\ /\  \uhdt(\x,t)=\uh^n(\x)\  \forall t\in[t_{n-1},t_n), \text{ with }\uh^n\in\Ht, 
\ \forall 1\leq n\leq N_T\Big\},\\
 \HDDdt&=\Big\{u_{h,\dt,\DD}\in L^1(Q_T)\ /\  u_{h,\dt,\DD}(\x,t)=u_{h,\DD}^n(\x)\  \forall t\in[t_{n-1},t_n), \\
 &\hspace*{8.5cm} \text{ with }u_{h,\DD}^n\in\HDD, 
\ \forall 1\leq n\leq N_T\Big\}.
 \end{align*}
 We still keep the notation $\nabla^h$ to define the approximate gradient of $\uhdt\in \Htdt$: 
 $$
 \nabla^h\uhdt(\x,t)=\nabla^h u_h^n(\x)\ \forall  t\in[t_{n-1},t_n).
 $$
 Therefore, for all $\uhdt\in \Htdt$, we have $\nabla^h\uhdt\in(\HDDdt)^2$.
  Furthermore, we introduce the following reconstructions
\begin{subequations}\label{defudth}
 \begin{align}
 & u_{h,\dt,\M}(\x,t)=u_{h,\M}^n(\x)=\sumpri\uk^n\boldsymbol{1}_\k(\x),&\forall t\in[t_{n-1},t_n),\\
 &u_{h,\dt,{\overline \Mie}}(\x,t)=u_{h,{\overline \Mie}}^n(\x)=\sumdua\uke^n\boldsymbol{1}_\ke(\x),&\forall t\in[t_{n-1},t_n).
\end{align}
\end{subequations}

 We may now introduce some norms on the functional spaces $\Ht$ and  $\Htdt$. For a discrete solution $\ut\in\Rt$,  we define $ \left|\ut\right|_{p,\T}$ for 
 $1\leq p\leq\infty$ by 
 \begin{eqnarray*}
 \left|\ut\right|_{p,\T}^p&=&\left(\frac{1}{2}\sumpri\mk\left|\uk\right|^p+\frac{1}{2}\sumdua \mke\left|\uke\right|^p
\right)^{1/p}\\
\left|\ut\right|_{\infty,\T}&=&\max\left(\max_{\k\in{ \M}}|\uk|, \max_{\ke\in{\overline \Mie}}|\uke|\right).
\end{eqnarray*}
It permits to define discrete $W^{1,p}$-norms ($1\leq p\leq +\infty$) and a discrete $W^{-1,1}$-norm on $\Ht$. For all $\uh\in\Ht$, we set 
\begin{eqnarray*}
 \left\|\uh\right\|_{1,p,\T}&=&\left(\left|\ut\right|_{p,\T}^p+ \left\|\nabla^h\uh\right\|_{p}^p\right)^{1/p},\quad \forall 1\leq p<+\infty, \\
 \|\uh\|_{1,\infty,\T}&=&\left|\ut\right|_{\infty,\T}+ \left\|\nabla^h\uh\right\|_{\infty},\\
  \|\uh\|_{1,\infty\star,\T}&=&\|\uh\|_{1,\infty,\T}+\left\llbracket\Pet \ut,\ut\right\rrbracket_\T^{\frac12},\\
 \|\uh\|_{-1,1,\T}&=&\max \bigg\{\left \llbracket\vt,\ut\right\rrbracket_\T,\forall \vh\in \Ht 
 \text{ verifying } \|\vh\|_{1,\infty\star,\petitt}\leq 1\bigg\}.
\end{eqnarray*}
Let us just remark that, as $\nabla^h\uh$ is a piecewise constant function on diamonds, we have:
$$
\left\|\nabla^h\uh\right\|_{p}^p=\sumdiam \md |\gradD\ut |^p\quad \forall 1\leq p<+\infty \mbox{ and }
\left\|\nabla^h\uh\right\|_{\infty}=\max_{\Ds\in\DD}|\gradD\ut|.
$$

Then, we define some discrete $L^q(0,T; W^{1,p}(\Omega))$ ($1\leq p,q<+\infty$), $L^\infty(0,T;W^{1,\infty}(\O))$ and $L^\infty(0,T;L^p(\O))$-norms on $\Htdt$. 
For all $\uhdt\in\Htdt$, we set:
\begin{eqnarray*}
\left\|\uhdt\right\|_{q;1,p,\T}&=&\left(\ssum_{n=1}^{N_T}\dt  \left\|\uh^{n}\right\|_{1,p,\T}^q\right)^{1/q},\quad \forall 1\leq p,q<+\infty, \\
\left\|\uhdt\right\|_{\infty;1,\infty,\T}&=&\underset{n\in\{1,\cdots,N_T\}}{\max} \left\|\uh^n\right\|_{1,\infty,\T},\\
\left\|\uhdt\right\|_{\infty;0,p,\T}&=&\underset{n\in\{1,\cdots,N_T\}}{\max}\left(\frac12\sumpri\mk|\uk^n|^p+\frac12\sumdua\mke|\uke^n|^p \right)^{1/p},\quad \forall 1\leq p<+\infty.
\end{eqnarray*}

\subsection{Main results}\label{ssec:main}
The numerical analysis of the scheme strongly relies on a discrete version of the 
energy/energy dissipation relation~\eqref{eq:EI}. In order to make it explicit, let us introduce the discrete counterpart $(\bbE_\T^n)_{n\geq 0}$ of the free energy $\bbE$ defined by \eqref{eq:E}:
$$
\bbE_\T^n = \llbracket H(u_\T^{n}), 1_\T \rrbracket_\T  +  \llbracket  V_\T , u_\T^{n} \rrbracket_\T, \qquad \forall n \ge 0,
$$
and the discrete counterpart $(\bbI^n_\T)_{n\geq 1}$ of the dissipation $\bbI$ defined by \eqref{eq:I}:
\be\label{eq:bbITn}
\bbI^n_\T = T_\DD \left(u_\T^n; g_\T^n, g_\T^n \right), \qquad \forall n \ge 1.
\ee

The first main result of our paper is the existence of a positive solution to the nonlinear scheme \eqref{scheme}; it is stated  in Theorem~\ref{thm:main1}. The mesh is given and 
fulfills the very permissive requirements of Section~\ref{sec-meshes}. Our nonlinear scheme~\eqref{scheme} yields a nonlinear system of 
algebraic equations. The fact that this system admits a solution is not obvious and is ensured by Theorem~\ref{thm:main1}. 
The proof strongly relies on the fact that the scheme fulfills a discrete entropy/dissipation relation. 

\begin{thm}[Existence of a discrete solution]\label{thm:main1}
For all $n\ge 0$, there exists a solution $u_\T^{n+1} \in \left(\R_+^\ast\right)^\T$ to the nonlinear system~\eqref{scheme} that 
satisfies the discrete entropy/entropy dissipation estimate
\be\label{eq:EI_D1}
\frac{\bbE^{n+1}_\T - \bbE^{n}_\T}\dt + \; \bbI_\T^{n+1}  \leq 0, \qquad \forall n \ge 0.
\ee
\end{thm}

Once the existence of $u_\T^{n+1}$ at hand for all $n \ge 0$, we can reconstruct the approximate solutions 
$u_{h,\dt}$, $u_{h,\dt, \M}$, and $u_{h,\dt, \M^*}$. The convergence of these approximate solutions towards 
a weak solution when the mesh size and the time step tend to 0 is then a very natural question. 
This question is addressed in Theorem~\ref{thm:main2}. 

In what follows, $\left(\T_m\right)_{m\ge 1} = \left(\ov \M_m, \ov{\M_m^*}\right)_{m\ge1}$ denotes a sequence of 
admissible discretization of $\O$ and $\left(\DD_m\right)_{m\ge1}$ denotes the corresponding diamond mesh. 
We assume that the
\be\label{eq:mesh-reg_m}
\size(\T_m) \underset{m\to\infty}\longrightarrow 0, \quad \text{whereas}\quad 
\underset{m\to\infty}{\textrm{limsup}} \max_{\Ds \in \DD_m} \max \left(\theta_\Ds,\widetilde\theta_\Ds\right) \leq \theta^\star,
\ee
the regularity factors $\theta_\Ds$ and $\widetilde \theta_\Ds$ being defined by~\eqref{eq:thetaD}. 

Concerning the time discretization, we consider a sequence $\left({N_{T,m}}\right)_{m\ge1}$ of positive integers tending to $+\infty$, 
and we denote by $\left(\dt_m\right)_{m\ge1} = \left(\frac{T}{N_{T,m}}\right)_{m\ge1}$ the corresponding sequence of time steps. 
For technical reasons that will appear later on, and even though this condition does not seem to be mandatory from a practical point of view, 
we have to make the assumption that there exists some constant $\ctel{CFL}>0$ such that 
\be\label{eq:CFL}
\dt_m \geq \cter{CFL} \size(\T_m), \qquad \forall m\ge 1.
\ee 

The existence of a discrete solution $u_{\T_m}^{n+1}$ to the scheme~\eqref{scheme} for all $n\in \{0, \dots, N_{T,m}-1\}$ 
and all $m \ge 1$ stated in Theorem~\ref{thm:main1} allows us to define the approximate solutions 
$u_{h_m, \dt_m}$, $u_{h_m, \dt_m, \M_m}$, and $u_{h_m, \dt_m, \ov{\M^*_m}}$ for all $m\ge 1$. 
The next theorem ensures that, up to a subsequence, the sequences of approximate solution converge 
towards a weak solution of the problem~\eqref{pb}.

\begin{thm}[Convergence towards a weak solution]\label{thm:main2}
Assume that \eqref{eq:mesh-reg_m} and \eqref{eq:CFL} holds. Then there exists a weak solution $u$ in the 
sense of Definition~\ref{Def:weak} such that, 
up to a subsequence, 
$$
u_{h_m, \dt_m, \M_m} \underset{m\to\infty}\longrightarrow u, \qquad
u_{h_m, \dt_m, \M_m^*} \underset{m\to\infty}\longrightarrow u, \quad \text{and} \quad 
u_{h_m, \dt_m} \underset{m\to\infty}\longrightarrow u
 \quad \text{in $L^p(0,T;L^1(\O))$}
$$
for all $p \in [1,\infty)$.
\end{thm}
Further convergence properties are established during the proof of Theorem~\ref{thm:main2}.
We don't make them explicit here in order to minimize the notations and to improve the readability of the paper. 
We refer to Section~\ref{ssec:compact} for refined statements. 

\begin{rem}
The convergence of the scheme is only assessed up to a subsequence in Theorem~\ref{thm:main2}. 
This comes from the fact that the uniqueness of weak solutions in the sense of Definition~\ref{Def:weak} 
is still an open problem even for initial data $u_0$ belonging to $L^2(\O)$. However, we conjecture 
that for $u_0$ being such that $H(u_0)$ belongs to $L^1(\O)$, 
the weak solutions in the sense of Definition~\ref{Def:weak} are renormalized solutions 
(see for instance~\cite{BP05}). Uniqueness should follow, implying the convergence of the whole sequence. 
\end{rem}

The main goal of the paper is to prove Theorems~\ref{thm:main1} and \ref{thm:main2}. 
The proof is articulated as follows: In Section~\ref{sec:existence}, we 
derive some estimates on the discrete solution. These {\em a priori} estimates allow us to show that the nonlinear 
system originating from the~\eqref{scheme} admits (at least) one solution, as claimed in Theorem~\ref{thm:main1}. 
Most of the estimates derived in Section~\ref{sec:existence} are uniform w.r.t. $m$. This provides enough compactness 
on the approximate solutions to pass to the limit $m\to\infty$ in Section~\ref{sec:convergence}.

\section{Energy and dissipation estimates, existence of a solution to the scheme}\label{sec:existence}

\subsection{{\em a priori} estimates}\label{ssec:estimates2}

The first statement of this section is devoted to what we call the fundamental estimates, that are discrete counterparts of the conservation of mass and 
of the energy/dissipation relation~\eqref{eq:EI}. All the further {\em a priori} estimates on the discrete solution are based on these two estimates. 

\begin{prop}[fundamental estimates]\label{prop:base}
Let $\left(u_\T^n\right)_{n\ge1}$, with $u_\T^n \in \left(\R_+^*\right)^\T$ for all $n\geq 0$, be  a solution to the scheme~\eqref{scheme} corresponding to the initial data $u_0$. Then,
\begin{enumerate}[(i)]
\item the mass is conserved along time, i.e., 
\be\label{eq:mass}
\int_\O u_{h}^n \d\x = \llbracket \ut^n, 1_\T \rrbracket_\T = \int_\O u_0 \d\x, \qquad \forall n \ge 0,
\ee
\item the discrete free energy is dissipated along time, i.e., 
\be\label{eq:EI_D2}
\frac{\bbE^{n+1}_\T - \bbE^{n}_\T}\dt + \; \bbI_\T^{n+1} + \kappa\,\llbracket \Pet g_\T^{n+1}, g_\T^{n+1}\rrbracket_\T  \leq 0, \qquad \forall n \ge 0.
\ee
\end{enumerate}
Moreover, the discrete free energy is decaying along time and is bounded: 
\be\label{en_decay}
0 \leq \bbE_\T^{n+1} \leq \bbE^n_\T \leq \bbE_\T^0 \leq \int_\O H(u_0) \d\x  + \|V\|_\infty \|u_0\|_{L^1(\O)}
\ee
and the ``integrated over time'' dissipation is also bounded:
\be\label{dissipation_control}
0\leq \sum_{n=1}^{N_T}\Delta t \bbI_\T^{n}\leq  \sum_{n=1}^{N_T}\Delta t (\bbI_\T^{n}+\kappa\,\llbracket \Pet g_\T^{n}, g_\T^{n}\rrbracket_\T )\leq\int_\O H(u_0) \d\x  + \|V\|_\infty \|u_0\|_{L^1(\O)}.
\ee
\end{prop}
\begin{proof}
Equation~\eqref{eq:mass} is obtained directly by choosing $\psi_\T = 1_\T$ in \eqref{sch_formcompacte}. In order to get Estimate~\eqref{eq:EI_D2}, 
it suffices to take $\psi_\T = g_\T^{n+1}$ in \eqref{sch_formcompacte} and to remark that, because of the convexity of $u \mapsto H(u) + uV$, one has 
$\llbracket \ut^{n+1} - \ut^n, g_\T^{n+1} \rrbracket_\T \ge \bbE_\T^{n+1} - \bbE^n_\T.$
Inequality \eqref{en_decay} is just a consequence of \eqref{eq:EI_D2}, of the nonnegativity of the dissipation and of the penalization term and of Jensen's inequality.  By summation of \eqref{eq:EI_D2} over $n$, we deduce \eqref{dissipation_control}.
\end{proof}

The goal of the remaining part of this section is to take advantage of the fundamental estimates 
of Proposition~\ref{prop:base} to derive some further estimates to be used in the numerical analysis. 
As in the continuous framework, the energy/dissipation estimate~\eqref{eq:EI_D2} is used in order to estimate the discrete counterpart 
of the Fisher information. But in the discrete framework, the chain rule appearing in~\eqref{eq:Fisher-chainrule} does not longer hold, 
and we have to manipulate several objects related to the Fisher information. 
The last goal of this section is to prove that, for a fixed grid and a fixed time step, the discrete solutions $\left(u_\T^n\right)_{n\ge1}$ 
is uniformly bounded away from 0, cf. Lemma~\ref{lem:pos}. 

Define the discrete fields $\xi_\T^n= \sqrt{u_\T^n}$ which play a key role as in the continuous level. 
In order to relate different discrete counterparts of the Fisher information, we first have to derive some properties on the 
local diffusion matrices $\bbA^\Ds$, $\Ds \in \DD$.

Let $\Ds \in \DD$, then we define the diagonal matrix $\bbB^\Ds$ by 
\be\label{eq:bbB}
\bbB^\Ds = \begin{pmatrix} B^\Ds_\sig & 0 \\ 0 & B^\Ds_{\sige}\end{pmatrix} =
\begin{pmatrix} \left|A^\Ds_{\sig,\sig}\right| + \left|A^\Ds_{\sig,\sige}\right| & 0 \\ 0 & 
\left|A^\Ds_{\sige,\sige}\right| + \left|A^\Ds_{\sig,\sige}\right| 
\end{pmatrix}.
\ee
For all $\bw = (w_\sig, w_{\sige})^T \in \R^2$ and all $\Ds \in \DD$, there holds
$$
\bbA^\Ds \bw \cdot \bw \leq \bbB^\Ds \bw \cdot \bw \leq \|\bbA^\Ds\|_1 |\bw|^2, 
$$
where $\|\cdot\|_q$ is the usual matrix $q$-norm and $|\cdot|$ is the Euclidian norm on $\R^2$.
It follows from the equivalence of the matrix $1$- and $2$- norms on the finite dimensional space $\R^{2\times2}$ that 
$$
\|\bbA^\Ds\|_1 |\bw|^2 \leq \gamma \|\bbA^\Ds\|_2 |\bw|^2 
\leq \gamma \text{Cond}_2(\bbA^\Ds)  \bbA^\Ds \bw \cdot \bw, \qquad \forall \bw \in \R^2
$$
for some $\gamma\ge 1$.
Therefore, in view of~\eqref{eq:condAD}, we get the existence of $\ctel{cte:ABequiv}$ depending only on $\theta^\star$, $\lambda_m$ and $\lambda^M$ 
such that 
\be\label{eq:ABequiv}
\bbA^\Ds \bw \cdot \bw \leq \bbB^\Ds \bw \cdot \bw \leq \cter{cte:ABequiv} \; \bbA^\Ds \bw \cdot \bw, \qquad \forall \Ds \in \DD, \; \forall \bw \in \R^2. 
\ee

We introduce now a discrete counterpart of $\int_\Omega u\Lambda\nabla \log u\cdot \nabla \log u$, it is the quantity $\widehat\bbI_\T^{n}$  defined by
\be\label{eq:whI}
\widehat \bbI_\T^{n+1} = \sum_{\Ds \in \DD}  r^\Ds(\ut^{n+1}) \; \delta^\Ds \log(u_\T^{n+1}) \cdot \bbB^\Ds \delta^\Ds \log(u_\T^{n+1}), 
\qquad n \ge 0. 
\ee
Let us first relate this quantity to a discrete Fisher information.
\begin{lem}\label{lem:Fisher}
For all $n \ge 0$, there holds
$$
\left\|\gradDD \xi_\T^{n+1} \right\|_{\bLambda, \DD}^2
\leq   \widehat \bbI_\T^{n+1}.
$$
\end{lem}
\begin{proof}
Thanks to the first inequality of~\eqref{eq:ABequiv}, one has
\be\label{eq:Fisher-00}
\left(\gradDD \xi_\T^{n+1},\gradDD \xi_\T^{n+1}\right)_{\bLambda, \DD}
\leq \sum_{\Ds \in \DD} \delta^\Ds \xi_\T^{n+1} \cdot \bbB^\Ds\delta^\Ds \xi_\T^{n+1}.
\ee
It results from the elementary inequality 
$$
\left|\sqrt{b} - \sqrt{a}\right| \leq \frac{\max(\sqrt a,\sqrt b)}2 \left|\log(b)-\log(a)\right|, \qquad \forall (a,b) \in \left(\R_+^*\right)^2,
$$
that for all $\Ds \in \DD$, one has 
\begin{align*}
\left|\xi_K^{n+1} - \xi_L^{n+1}\right| \leq &
\frac{\max(\xi_K^{n+1},\xi_L^{n+1},\xi_{K^*}^{n+1},\xi_{L^*}^{n+1})}2 \left|\log(u_{K}^{n+1})-\log(u_{L}^{n+1})\right|, \\
\left|\xi_{K^*}^{n+1} - \xi_{L^*}^{n+1}\right| \leq &
\frac{\max(\xi_K^{n+1},\xi_L^{n+1},\xi_{K^*}^{n+1},\xi_{L^*}^{n+1})}2 \left|\log(u_{K^*}^{n+1})-\log(u_{L^*}^{n+1})\right|.
\end{align*}
This yields 
$$
\delta^\Ds \xi_\T^{n+1} \cdot \bbB^\Ds\delta^\Ds \xi_\T^{n+1}
 \leq \frac{\max({u_K^{n+1}},{u_L^{n+1}},{u_{K^*}^{n+1}},{u_{L^*}^{n+1}})}4 
 \delta^\Ds \log({\ut^{n+1}}) \cdot \bbB^\Ds\delta^\Ds \log({\ut^{n+1}}), 
 $$
 and since 
 $\max({u_K^{n+1}},{u_L^{n+1}},{u_{K^*}^{n+1}},{u_{L^*}^{n+1}}) \leq 4 r^\Ds(\ut^{n+1})$, one gets 
\be\label{eq:Fisher-01}
 \delta^\Ds \xi_\T^{n+1} \cdot \bbB^\Ds\delta^\Ds \xi_\T^{n+1} \leq r^\Ds(\ut^{n+1})  \delta^\Ds \log({\ut^{n+1}}) \cdot \bbB^\Ds\delta^\Ds \log({\ut^{n+1}}), \quad \forall \Ds \in \DD.
\ee
In order to conclude the proof of Lemma~\ref{lem:Fisher}, it only remains to combine~\eqref{eq:Fisher-00} and 
\eqref{eq:Fisher-01}.
\end{proof}

We now want to get a bound on $\widehat \bbI^{n+1}_\T$ in order to deduce some bound on $\Vert \nabla^\DD \xi_\T^{n+1}\Vert_{{\mathbf \Lambda},\DD}$. Therefore, we first need to establish an estimate on the discrete reconstruction by diamond $r^\DD[u_\T^{n+1}]$.

\begin{lem}\label{lem:mass-ud}
Let $r^\DD[\ut^{n+1}] \in \R^\DD$ be defined by~\eqref{eq:rD}.
There exists $C>0$, depending only on $\O$, $\lambda_m$, $\lambda^M$ and $\theta^\star$ such that 
\be\label{eq:mass-ud}
\sum_{\Ds \in \DD} \md \,r^\Ds(\ut^{n+1}) \leq C\, (1+{\rm size}(\T))\int_\O u_0 \d\x + C\, {\rm size}(\T)\, \widehat \bbI_\T^{n+1}, \qquad \forall n \ge 0.
\ee
\end{lem}
\begin{proof}
The definition~\eqref{eq:rD} implies that 
\be\label{eq:mass-ud1}
\sum_{\Ds \in \DD} \md \,r^\Ds(\ut^{n+1}) = (T^{n+1} + T^{{n+1,*}})/4, 
\ee
where we have set 
$$
T^{n+1} = \sum_{\Ds \in \DD} \md (u_K^{n+1}+u_L^{n+1})
\quad \text{and}\quad
T^{{n+1,*}}= \sum_{\Ds \in \DD} \md (u_{K^*}^{n+1}+u_{L^*}^{n+1}), 
\qquad \forall n \ge 0.
$$
The terms $T^{n+1}$ and $T^{{n+1,*}}$ can be rewritten
\begin{align*}
T^{n+1}  =& \; \sum_{K \in \M}u_K^{n+1}  \sum_{\Ds \in \DD_K} \md 
+\sum_{L \in \dr\M}u_L^{n+1} {\rm m}_{\Ds_L} 
= T_\M^{n+1} + T_{\p\M}^{n+1},\\
T^{{n+1,*}} = & \; \sum_{K^* \in \ov{\M^*}} u_{K^*}^{n+1} \sum_{\Ds \in \DD_{K^*}} \md,
\end{align*}
where ${\Ds_L}$ denotes the unique diamond cell associated to the primal boundary cell $L \in \p\M$.
The terms $T_\M^{n+1}$ and $T^{{n+1,*}}$ can be estimated thanks to the regularity of the mesh~\eqref{eq:reg-meas} by 
$$
T_\M^{n+1} + T^{{n+1,*}} \leq  \theta^\star \left( \sum_{K \in \M} \mk u_K^{n+1} + \sum_{K^* \in \ov{\M^*}} u_{K^*}^{n+1} \mke \right).
$$
We deduce from~\eqref{eq:mass}
that 
$$
T_\M^{n+1} + T^{{n+1,*}} \leq 2 \theta^\star \int_\O u_0 \d\x.
$$

Let us now focus on the term $T_{\p\M}^{n+1}$. The area of the diamond cell $\Ds$ corresponding to a 
boundary edge $\sigma \subset \p\O$ (or equivalently to a primal boundary cell $L\in \p\M$) can be estimated by 
$$
\md \leq \frac12 {\rm m}_\sig {\rm size}(\T). 
$$
Therefore, we get that 
$$
T_{\p\M}^{n+1} \leq\frac12 {\rm size}(\T) \| \gamma_{\p\M} \ut^{n+1}\|_{L^1(\p\O)} =
 \frac12{\rm size}(\T)\left\| \gamma_{\p\M} \xi_\T^{n+1}\right\|^2_{L^2(\p\O)},
$$
where $\gamma_{\p\M} \ut(\x)=\sumpriext\ul\boldsymbol{1}_{\l}(\x),$ $ \forall \x \in \p\O. $
The trace inequality stated in Theorem~\ref{thmtracebis} gives 
with $\vt = \xi_\T^{n+1}$
$$
T_{\p\M}^{n+1} \leq C\;{\rm size}(\T)\, \left(|\xi_\T^{n+1}|_{2,\T}^2+ \left\|\nabla^h\xi_h^{n+1}\right\|_{2}^2\right).
$$
Thanks to \eqref{eq:elliptic} and the regularity of the mesh \eqref{eq:thetaD} and \eqref{eq:theta_star}, 
there exists $C>0$ depending only on $\lambda_m$, $\lambda^M$ and $\theta^\star$ such that
\begin{equation}\label{eq:normegradsqrtut}
 \left\|\nabla^h\xi_h^{n+1}\right\|_{2}\leq C\left\|\gradDD \xi_\T^{n+1} \right\|_{\bLambda, \DD}
\end{equation}
Since $|\xi_\T^{n+1}|_{2,\T}^2= |\ut^{n+1}|_{1,\T}$, Proposition~\ref{prop:base} and Lemma \ref{lem:Fisher} provide that 
$$
T_{\p\M}^{n+1}\leq C\;{\rm size}(\T) \, \left( \int_\O u_0 \d\x+ \widehat \bbI_\T^{n+1}\right).
$$
\end{proof}

Thanks to \eqref{eq:mass-ud}, it is now possible to relate  $\widehat \bbI^{n+1}_\T$ to $ \bbI^{n+1}_\T$.

\begin{lem}\label{lem:Fisher-0}
There exist $C>0$ and $h^\star>0$ depending only on $u_0$, $V$, $\lambda_m$, $\lambda^M$ and $\theta^\star$ such that 
\be\label{eq:Fisher-0}
\widehat \bbI_\T^{n+1} \leq C\,\left(1+ \bbI_\T^{n+1}\right), \qquad \forall n \ge0, \quad \text{if  ${\rm size}(\T) \leq h^\star$.}
\ee
\end{lem}
\begin{proof}
Bearing in mind the definition \eqref{eq:bbITn} of the dissipation $\bbI_\T^{n+1}$, we deduce thanks to \eqref{eq:ABequiv} that
$$
\bbI_\T^{n+1} \leq  \sum_{\Ds \in \DD}  r^\Ds(\ut^{n+1}) \; \delta^\Ds g_\T^{n+1} \cdot \bbB^\Ds \delta^\Ds g_\T^{n+1} 
\leq \cter{cte:ABequiv}\, \bbI_\T^{n+1}, \qquad \forall n \ge 0.
$$
But, as $g_\T^{n+1}=\log u_\T^{n+1} +V_\T$, the elementary inequality $(a+b)^2\leq 2 (a^
2+b^2)$ implies that 
\be\label{eq:I-whI2}
\widehat \bbI_\T^{n+1}
\leq  2 \, \cter{cte:ABequiv}\, \bbI_\T^{n+1} + 2 \sum_{\Ds \in \DD} r^\Ds(\ut^{n+1}) \delta^\Ds V_\T \cdot \bbB^\Ds \delta^\Ds V_\T.
\ee
It follows from the regularity of $V$ and from the regularity of the mesh~\eqref{eq:theta_star} that
$$
 0 \leq \delta^\Ds V_\T \cdot \bbB^\Ds \delta^\Ds V_\T \leq \|\grad V\|_\infty^2  
 \left(B^\Ds_\sig \msige^2 + B^\Ds_{\sige}\msig^2\right) 
 \leq  C
 \|\grad V\|_\infty^2 \md, 
 \qquad \forall \Ds \in \DD
$$
for some $C$ depending only on $\lambda^M$ and $\theta^\star$.
Therefore, we deduce from Lemma~\ref{lem:mass-ud} that 
\be\label{eq:I-whI3}
\sum_{\Ds \in \DD} r^\Ds(\ut^{n+1})\delta^\Ds V_\T \cdot \bbB^\Ds \delta^\Ds V_\T 
\leq C \sum_{\Ds \in \DD} \md r^\Ds(u_\T^{n+1}) \leq C(1+{\rm size}(\T) \widehat \bbI_\T^{n+1}).
\ee
We infer from~\eqref{eq:I-whI2} and \eqref{eq:I-whI3} that if  ${\rm size}(\T)$ is small enough,~\eqref{eq:Fisher-0} holds.
\end{proof}

We have at hand the necessary tool to address the uniform positivity of the solutions. The next lemma states that the discrete solutions remain 
bounded away from $0$ by a small quantity $\eps>0$ depending on the data of the continuous problem and on the discretization parameters.
This information is of great importance since, because of the singularity of the $\log$, the nonlinear functional corresponding to the scheme is not continuous on the boundary of $\left(\R_+^*\right)^\T$.
The proof is inspired from the ones of \cite[Lemma 3.10]{CG16_MCOM} and \cite[Lemma 3.7]{CG2016}. We sketch it here to 
highlight how we overpass the difficulties related to the fact that there is a limited communication between the primal and dual meshes.

\begin{lem}\label{lem:pos}
There exists $\eps>0$ depending on the data $u_0$ and $V$, 
on the mesh $\T$, on the time step $\dt$, and on the stabilization parameter $\kappa$, such that
\be\label{eq:pos}
u_K^{n+1} \geq \eps\quad\text{and}\quad  u_{\ke}^n\geq \eps, \qquad \forall K \in \ov \M, \, \forall \ke \in \ov{\M^\ast}, \, \forall n \ge 0.
\ee
\end{lem}
\begin{proof}
In this proof as elsewhere in the paper, the generic constants $C$ only depend on the data of the continuous 
problem and on the regularity bound $\theta^\star$ for the mesh. In order to highlight the dependency of 
a quantity with respect to the mesh or to the time step, we use subscripts. For instance $C_\dt$ may depend on $\dt$, whereas 
$C_\T$ may depend on the mesh and 
$C_{\T,\dt}$ may depend on the mesh and on the time step.

Owing to~\eqref{eq:mass}, there exists $M_0 \in \M \cap \ov{\M^*}$ such that 
$$
u_{M_0}^{n+1}  \ge \frac1{{\rm m}_\O} \int_\O u_0\d\x >0, 
$$
implying in particular that 
\be\label{eq:pos1}
\log(u_{M_0}^{n+1}) \ge - C \quad  \text{and} \quad r^\Ds(\ut^{n+1}) \ge  \frac1{4{\rm m}_\O} \int_\O u_0\d\x >0 \quad\text{for all}\quad \Ds \in \DD_{M_0}.
\ee
On the other hand, it follows from \eqref{dissipation_control} that
$\bbI_\T^{n+1} \leq C_\dt$.
Together with Estimate~\eqref{eq:Fisher-0}, this provides that 
\be\label{eq:pos2}
\widehat \bbI_\T^{n+1} \leq C_\dt.
\ee
Assume for instance that $M_0=K_0 \in \M$ (the case $M_0 \in \ov{\M^*}$ is similar). Since $B^\Ds_\sig > \frac1C$ for all $\Ds \in \DD$, 
we deduce from \eqref{eq:pos1}--\eqref{eq:pos2} that $\log(u_{K_1}^{n+1}) \ge - C_\dt$ for all neighboring cell ${K_1} \in \M$ such that $K_0|K_1 \in \Ee$ and for all $L \in \p\M\cap \p K_0$. It follows from a simple induction based on the reproduction of this argument (see \cite[Lemma 3.10]{CG16_MCOM} or \cite[Lemma 3.7]{CG2016} for details) that 
\be\label{eq:pos-primal}
\log\left(u_K^{n+1}\right) \geq -C_{\dt, \T}, \qquad \forall K \in \ov \M.
\ee

Let $K^* \in\ov { \M^*}$, then there exists $K \in \M$ such that m$(\k \cap \ke) >0$. 
Thanks to the penalization term in Estimate~\eqref{eq:EI_D2},
one has that
$$
\left(g_{K^*}^{n+1} - g_K^{n+1}\right)^2 \leq C_{\T,\dt}.
$$
The regularity of $V$ implies that
$$
\left(\log\left(u_{K^*}^{n+1}\right) - \log\left(u_K^{n+1}\right)\right)^2 \leq C_{\T,\dt},
$$
hence, using  \eqref{eq:pos-primal}, one gets that
\be\label{eq:pos-dual}
\log\left(u_{K^*}^{n+1}\right) \geq -C_{\dt, \T}, \qquad \forall \ke \in \ov {\M^*}.
\ee
The relation~\eqref{eq:pos} follows from \eqref{eq:pos-primal} and \eqref{eq:pos-dual}.
\end{proof}

\subsection{Existence of a solution to the scheme}

The numerical scheme~\eqref{scheme} amounts at each time step $n \ge 0$ to solve 
a nonlinear system $\boldsymbol{\mathcal F}^n(\ut^{n+1}) = \0$. The existence of a solution $\ut^{n+1}$ 
is therefore non trivial. It is established in the following proposition. 
\begin{prop}\label{prop:existence}
For all $n \ge 0$, there exists (at least) one solution $u_\T^{n+1} \in (\R_+^\ast)^\T$ to the nonlinear system~\eqref{scheme}.
\end{prop}
The proof of Proposition \ref{prop:existence} relies on a topological degree argument~\cite{LS34,Dei85,EGGH98}. 
The key point is that, owing to Lemma~\ref{lem:pos}, one can restrain our search for the solution on a compact subset of 
$\left(\R_+^*\right)^\T$ on which the functional $ \boldsymbol{\mathcal F}^n$ is (uniformly) continuous, making 
Leray-Schauder's theorem applicable. 
We do not detail the proof here since it is very close to the one of~\cite[Proposition 3.8]{CG2016}.

\section{Convergence w.r.t. discretization parameters}\label{sec:convergence}

\subsection{Compactness of the approximate solutions}\label{ssec:compact}

Thanks to Proposition~\ref{prop:existence}, we have at hand discrete solutions $\left(\ut^{n+1}\right)_{n\ge 0}$ corresponding 
to all the time steps, and thus the corresponding reconstructions 
$u_{h,\Delta t} \in \Htdt$, $u_{h,\DD,\Delta t}$ as defined in Section~\ref{ssec:functspaces}. We also define $\xi_{h,\Delta t} \in \Htdt$ based on $\xi_\T^n = \sqrt{\ut^n}$ for all $n \in \{0,\dots, N_T\}$.

Thanks to the estimates established in Section~\ref{ssec:estimates2}, we can obtain 
some further estimates satisfied by the discrete reconstructions. These estimates, stated in Lemma \ref{lem:estsolapp}, will then be used 
to deduce some compactness properties of sequences of  approximate solutions.  

\begin{lem}\label{lem:estsolapp}
\begin{enumerate}[(i)]
\item There exists $C$ depending only on $u_0$ and $V$ such that 
\be\label{est:equi-int}
\int_\O H({u_{h,\Delta t}}(\x,t)) \d\x \leq \frac12\left( \int_\O H({u_{h,\Delta t, \M}}(\x,t)) \d\x  + \int_\O H({u_{h,\Delta t, \M^*}}(\x,t)) \d\x \right)\leq  C, \ \forall t \ge 0.
\ee
\item There exists $C$ depending only on $u_0$, $V$, $\lambda_m$, $\lambda^M$ and $\theta^\star$ such that, 
for ${\rm size}(\T)$ small enough, one has
\be
\label{eq:Fisher-sum}
 \sum_{n = 1}^{N_T} \dt\,\widehat \bbI^n_\T \leq C(1+T) \mbox{ and } \iint_{Q_T} |\grad^h \xi_{h,\dt}|^2 \d\x \d t\leq C(1+T).
\ee 
\item There exists $C$ depending only on $\theta^\star$, $u_0$, $V$, $\cter{CFL}$,
$\lambda_m$ and $\lambda^M$ such that
\be\label{est:u_D_L1}
\int_\O u_{h,\DD}^n \d\x \leq C(1+T), \qquad \forall n \in \{1,\dots, N_T\}.
\ee

\end{enumerate}
\end{lem}

\begin{proof}
The first inequality in \eqref{est:equi-int} is just a consequence of Jensen's inequality because $H$ is a convex function. 
Moreover, since we assumed that $V\geq 0$ and proved the positivity of the discrete solution, we have: 
\begin{multline*}
\int_\O H({u_{h,\Delta t}}(\x,t))\leq \frac12\left( \int_\O H({u_{h,\Delta t, \M}}(\x,t)) \d\x  + \int_\O H({u_{h,\Delta t, \M^*}}(\x,t)) \d\x \right) \\
\leq  \bbE^{n+1}_\T, \ \forall t \in ( t^n, t^{n+1}],
\end{multline*}
The last inequality in  \eqref{est:equi-int} is then a straightforward consequence of \eqref{en_decay}. The estimates in \eqref{eq:Fisher-sum} are deduced from \eqref{dissipation_control}, \eqref{eq:Fisher-0} and Lemma \ref{lem:Fisher}.

It remains to prove \eqref{est:u_D_L1}. From Lemma \ref{lem:mass-ud}, we have that 
$$
\int_\O u_{h,\DD}^n \d\x \leq C \left(1 + \size(\T) \,\widehat \bbI_\T^n \right), \qquad \forall n \ge 1.
$$
We infer from the first inequality of~\eqref{eq:Fisher-sum} that 
$
 \widehat \bbI_\T^n \leq {C(1+T)}/{\dt}
$
and the assumption~\eqref{eq:CFL} implies that 
$$
\size(\T) \,\widehat \bbI_\T^n \leq C(1+T), \qquad \forall n \ge 1.
$$
This concludes the proof of Lemma \ref{lem:estsolapp}.

\end{proof}

In order to get the compactness of a sequence of approximate solutions, it is also crucial to establish a discrete counterpart of a $L^1(0,T;W^{-1,1}(\O))$ estimate on the discrete time derivative. In what follows, we denote by 
$$\p_{t,\T} u_{h,\dt}^n = \left( \left(\frac{u_K^{n+1}-u_K^{n}}{\dt}\right)_{K\in\ov \M},  \left(\frac{u_{K^*}^{n+1}-u_{K^*}^{n}}{\dt}\right)_{K^*\in\ov{\M^*}}\right) \in \R^\T, 
\qquad \forall n \ge 0.$$
\begin{lem}\label{lem:u_t}
There exists $C$ depending only on $V$, $u_0$, $T$, $\kappa$, $\theta^\star$, $\cter{CFL}$, $\lambda_m$ and $\lambda^M$ such that 
$$
\sum_{n=0}^{N_T-1}\dt \left\|\p_{t,\T} u_{h,\dt}^n \right\|_{-1,1,\T}  \leq C.
$$
\end{lem}
\begin{proof}
We proceed as in \cite[Lemma 3.4]{CKM15}. It follows from \eqref{sch_formcompacte} that for all $\psi_\T \in \R^\T$, one has 
$$
\llbracket \p_{t,\T} u_{h,\dt}^n, \psi_\T \rrbracket_\T = -  T_\DD(u_\T^{n+1}, g_\T^{n+1}, \psi_\T) 
- \kappa \llbracket \mathcal P(g_\T^{n+1}), \psi_\T\rrbracket_\T, \qquad \forall n \ge 0.
$$
The application $(g_\T^{n+1}, \psi_\T) \mapsto T_\DD(u_\T^{n+1}, g_\T^{n+1}, \psi_\T) + 
\kappa\llbracket \mathcal P(g_\T^{n+1}), \psi_\T\rrbracket_\T$ is a scalar 
product on $\R^\T$, then it follows from Cauchy-Schwarz inequality that 
$$
\llbracket \p_{t,\T} u_{h,\dt}^n, \psi_\T \rrbracket_\T \leq 
\left( \bbI_\T^{n+1} +  \kappa\llbracket \mathcal P(g_\T^{n+1}), g_\T^{n+1}\rrbracket_\T \right)^{1/2}
\left( T_\DD(u_\T^{n+1}, \psi_\T, \psi_\T) +  \kappa\llbracket \mathcal P(\psi_\T), \psi_\T\rrbracket_\T \right)^{1/2} .
$$
We can estimate the term 
$T_\DD(u_\T^{n+1}, \psi_\T, \psi_\T)$ by 
$$
 T_\DD(u_\T^{n+1}, \psi_\T, \psi_\T) = \int_\O u_{h,\DD}^{n+1}Â \Lambda_{h,\DD} \grad^h \psi_h \cdot   \grad^h \psi_h \d\x
 \leq \lambda^M \left\|u_{h,\DD}^{n+1}\right\|_{L^1(\O)} {\|\grad^h \psi_h\|}^2_\infty.
$$
Thanks to \eqref{est:u_D_L1} and since $u_{h,\DD}^{n+1} \ge 0$, one gets that 
$$
 T_\DD(u_\T^{n+1}, \psi_\T, \psi_\T) \leq  C\; {\|\grad^h \psi_h\|}^2_\infty, \qquad \forall n \ge0, 
$$
therefore 
$$
\left( T_\DD(u_\T^{n+1}, \psi_\T, \psi_\T) + \kappa\llbracket \mathcal P(\psi_\T), \psi_\T\rrbracket_\T \right)^{1/2}  
\leq C  \|\psi_h\|_{1,\infty\star,\T}, \qquad \forall \psi_\T \in \R^\T.
$$
Since $\psi_\T$ can be chosen arbitrarily, this implies that 
$$
 \left\|\p_{t,\T} u_{h,\dt}^n \right\|_{-1,1,\T} \leq C  \left( \bbI_\T^{n+1} +  \kappa\llbracket \mathcal P(g_\T^{n+1}), g_\T^{n+1}\rrbracket_\T \right)^{1/2}, \qquad \forall n \in \{0,\dots, N_T-1\}.
$$
Thanks to Cauchy-Schwarz inequality once again, we obtain that 
\begin{align*}
\sum_{n=0}^{N_T-1}Â \dt \left\|\p_{t,\T} u_{h,\dt}^n \right\|_{-1,1,\T} 
\leq\; &\sqrt T \left(\sum_{n=0}^{N_T-1}Â \dt \left\|\p_{t,\T} u_{h,\dt}^n \right\|_{-1,1,\T}^2\right)^{1/2} \\
\leq \;&C  \left(\sum_{n=1}^{N_T}Â \dt \left(\bbI_\T^n + \kappa\llbracket \mathcal P(g_\T^{n}), g_\T^{n}\rrbracket_\T\right) \right)^{1/2}.
\end{align*}
One concludes the proof by using \eqref{dissipation_control}.
\end{proof}

Let $\left(\T_m\right)_{m\ge 1}$ be a sequence of meshes as in Section \ref{sec-meshes} such that ${\rm size}(\T_m)$ tends to 
$0$ as $m$ tends to $\infty$, and such that the regularity of the discretization $\T_m$ is uniformly bounded w.r.t. $m$, i.e., 
$$
1 \leq \theta_\Ds, \widetilde \theta_\Ds \leq \theta^\star, \qquad \forall \Ds \in \DD_m, \; \forall m \ge 1.
$$
Assume moreover that simultaneously the time step $\dt_m$ tends to $0$ as $m$ tends to $\infty$ while satisfying~\eqref{eq:CFL}.
The estimates stated in this section are uniform w.r.t. $m$. 
We deduce from Lemma~\ref{lem:estsolapp} that the sequences $\left(u_{h_m, \dt_m}\right)_{m\ge 1}$, 
$\left(u_{h_m, \dt_m,\M}\right)_{m\ge 1}$, and $\left(u_{h_m, \dt_m,\M^*}\right)_{m\ge 1}$ are  equi-integrable
in $L^1(Q_T)$ while uniformly bounded in $L^\infty(0,T;L^1(\O))$, 
hence there exists $u, u^{(1)}, u^{(2)} \in L^\infty(0,T;L^1(\O))$ such that, for all $p \in [1,\infty)$, 
the following convergence holds up to the extraction of an unlabeled subsequence:
\begin{subequations}\label{eq:convL1-weak}
\begin{align}
&u_{h_m, \dt_m} \underset{m\to\infty}\longrightarrow u \quad \text{weakly in}\; L^p(0,T;L^1(\O)), \label{eq:convL1-weak-h}\\
& u_{h_m, \dt_m, \M} \underset{m\to\infty}\longrightarrow u^{(1)} \quad \text{weakly in}\; L^p(0,T;L^1(\O)).\label{eq:convL1-weak-hM}\\
& u_{h_m, \dt_m, \M^*} \underset{m\to\infty}\longrightarrow u^{(2)} \quad \text{weakly in}\; L^p(0,T;L^1(\O)).\label{eq:convL1-weak-hM*}
\end{align}
\end{subequations}

Moreover, it follows from \eqref{eq:mass} and \eqref{eq:Fisher-sum} that 
\be\label{eq:L2H1-discrete}
\|\xi_{h,\Delta t}\|_{2;1,2,\T}\leq C.
\ee 
Thus, similarly to \cite[Proposition 5.3]{ABK2010} (which is strongly related to \cite[Lemma 3.8]{ABH_2007}), we get the existence of  $\xi \in L^2(0,T;H^1(\O))$ such that 
\be\label{eq:conv-xi-weak}
\xi_{h_m, \dt_m} \underset{m\to\infty}\longrightarrow \xi \quad \text{weakly in}\; L^2(Q_T), 
\qquad 
\grad^h \xi_{h_m, \dt_m} \underset{m\to\infty}\longrightarrow \grad \xi\quad \text{weakly in}\; L^2(Q_T)^2. 
\ee

Up to now, we only have weak convergence results, which are  not sufficient to pass to the limit in the scheme because of the nonlinearities. The purpose of the following proposition is to recover some strong convergence results. Our approach is based on the time-compactness toolbox presented in~\cite{ACM15_preprint} (see also~\cite{Gallouet_FVCA8} for a closely related approach).

\begin{prop}\label{prop:strong-conv} Up to the extraction of an unlabeled subsequence, 
\begin{subequations}\label{eq:convL1-strong}
\begin{align}
&u_{h_m, \dt_m}  \underset{m\to\infty}\longrightarrow u \quad \text{strongly in}\; L^p(0,T;L^1(\O)), \label{eq:convL1-strong-h}\\
& u_{h_m, \dt_m, \M_m} \underset{m\to\infty}\longrightarrow u \quad \text{strongly in}\; L^p(0,T;L^1(\O)),\label{eq:convL1-strong-hM}\\
& u_{h_m, \dt_m, \M^*_m} \underset{m\to\infty}\longrightarrow u \quad \text{strongly in}\; L^p(0,T;L^1(\O)),\label{eq:convL1-strong-hM*}\\
& u_{h_m, \dt_m, \DD_m} \underset{m\to\infty}\longrightarrow u \quad \text{strongly in}\; L^1(Q_T),\label{eq:convL1-strong-hD}
\end{align}
\end{subequations}
for all $p \in [1,\infty)$.
Moreover, let $\xi \in L^2(0,T,H^1(\O))$ be as in~\eqref{eq:conv-xi-weak}, then $\xi=\sqrt{u}$.
\end{prop}
\begin{proof}
The proof is divided into three steps.
We remove the subscripts $m$ for the ease of reading.
\smallskip

\noindent {\em Step 1.} The goal of this part is to make use on both $\left(u_{h,\dt,\M}\right)$ and $\left(u_{h,\dt,\M^*}\right)$ 
of the time-compactness criterion of~\cite[Theorem 3.9]{ACM15_preprint}. 

It follows directly from their definitions that 
$$
u_{h,\dt,\M} = \left(\xi_{h,\dt,\M}\right)^2, \qquad u_{h,\dt,\M^*} = \left(\xi_{h,\dt,\M^*}\right)^2.
$$
Thanks to discrete Poincar\'e-Sobolev Inequality~\cite{BCCHF15}, it follows from~\eqref{eq:L2H1-discrete} that 
 $$
\| \xi_{h,\dt,\M} \|_{L^2(0,T;L^p(\O))} \leq C_p, \qquad \| \xi_{h,\dt,\M^*} \|_{L^2(0,T;L^p(\O))} \leq C_p, \qquad \forall p\in [1,+\infty),
 $$
 for some $C_p$ depending only on $p$, on $\O$ and on the regularity $\theta^\star$ of the mesh and  therefore
  \be\label{eq:uL1Lp}
\| u_{h,\dt,\M} \|_{L^1(0,T;L^p(\O))} \leq C_p, \qquad \| u_{h,\dt,\M^*} \|_{L^1(0,T;L^p(\O))} \leq C_p, \qquad \forall p \in [1,\infty).
 \ee
On the other hand, the mass conservation \eqref{eq:mass} and the positivity of the solutions yield 
 \be\label{eq:uLinfL1}
\| u_{h,\dt,\M} \|_{L^\infty(0,T;L^1(\O))} \leq C, \qquad \| u_{h,\dt,\M^*} \|_{L^\infty(0,T;L^1(\O))} \leq C. 
\ee
It results from~\eqref{eq:uL1Lp}, \eqref{eq:uLinfL1} and from Riesz-Thorin interpolation theorem that 
 \be\label{eq:uLpLp}
\| u_{h,\dt,\M} \|_{L^p(Q_T)} \leq C_p, \qquad \| u_{h,\dt,\M^*} \|_{L^p(Q_T)} \leq C_p, \qquad 
\forall p \in [1,2), 
\ee
thus in particular for $p=3/2$.  The weak limits $u^{(1)}$, $u^{(2)}$ of $u_{h,\dt,\M}$ and $u_{h,\dt,\M^*}$ thus belong to $L^{3/2}(Q_T)$, 
and  the weak limits of $\xi_{h,\dt,\M}$ and $\xi_{h,\dt,\M^*}$ (up to the extraction of an unlabeled subsequence), denoted by $\xi^{(1)}$, $\xi^{(2)}$, belong to $L^{3}(Q_T)$ of . The functions $u_{h,\dt,\M}$ and $\xi_{h,\dt,\M}$, as well as $u_{h,\dt,\M^*}$ and $\xi_{h,\dt,\M^*}$, are thus in duality. 

We now want to apply \cite[Theorem 3.9]{ACM15_preprint} in order to show that $\xi^{(1)}=\sqrt{u^{(1)}}$ and $\xi^{(2)}=\sqrt{u^{(2)}}$. Therefore, we have to verify that the three assumptions {$\bf (A_x1)$}, {$\bf (A_x2)$} and {$\bf (A_x3)$} of  \cite{ACM15_preprint} are satisfied. 
\begin{enumerate}[(i)]
\item 
As a direct consequence of the arguments developed in the proofs of \cite[Lemma 3.8]{ABH_2007} or \cite[Proposition 4.3]{CKM15} and of Poincar\'e Sobolev embedding~\cite{BCCHF15}, any sequence $\left(v_{\T_m}\right)_m$ such that $\left\|v_{h}\right\|_{1,2,\T} \leq C$ 
is such that $v_{h,\M}$ and $v_{h,\Mie}$ converges strongly in $L^3(\O)$ (up to the extraction of an unlabeled subsequence). Therefore Assumption~{$\bf (A_x1)$} of~\cite{ACM15_preprint} is fulfilled.
\item The reconstructions $v_{h,\M}$ and $v_{h,\Mie}$ from $v_\T$ are piecewise constant, the functions being equal to nodal values of 
$v_\T$ a.e. in $\O$, then Assumption~{$\bf (A_x2)$} of~\cite{ACM15_preprint} is fulfilled by these reconstructions. Note here that this 
is not the case of the reconstruction $v_h$.
\item Let $\varphi \in C^\infty(\ov\O)$ and let us define $\varphi_\T$ by 
\be\label{eq:proj_phi}
\begin{cases}
\dsp \varphi_K = \frac1{\mk} \int_K \varphi \d\x & \text{for}\;  K \in \M,\\[5pt]
\dsp \varphi_L =  \frac1{\msig} \int_\sig \varphi \d\x& \text{for}\; L\equiv \sig \in \p\M,\\[5pt]
\dsp \varphi_{K^*} = \frac1{\mke} \int_{K^*} \varphi \d\x& \text{for}\; {K^*} \in \ov{\M^*}.
\end{cases}
\ee
Following the proof of \cite[Proposition 4.2]{CKM15}, there exists $C$ depending only on the regularity of the mesh $\theta^\star$ 
such that 
\be\label{eq:grad-Linf0}
\left\|\varphi_h \right\|_{1,\infty,\T} \leq C \|\grad \varphi\|_\infty, \qquad \forall \varphi \in C^\infty(\ov \O).
\ee
On the other hand, one can show that 
\be\label{eq:stab-Lip}
\left\llbracket\Pet \varphi_\T,\varphi_\T\right\rrbracket_\T \leq C \|\grad \varphi \|_\infty^2
\ee
for some $C$ depending only on the regularity $\theta^\star$ of the mesh. 
Therefore, 
\be\label{eq:grad-Linf}
\left\|\varphi_h \right\|_{1,\infty\star,\T} \leq C \|\grad \varphi\|_\infty, \qquad \forall \varphi \in C^\infty(\ov \O).
\ee
Assumption $\bf (A_x3)$ of \cite{ACM15_preprint} is thus fulfilled. 
\end{enumerate}
We can then make use of Lemma~\ref{lem:u_t} and apply~\cite[Theorem 3.9]{ACM15_preprint} 
to claim that $\xi^{(i)} = \sqrt{u^{(i)}}$ and that 
\be\label{eq:ae-conv12}
\begin{cases}
u_{h,\dt,\M} \underset{m\to \infty}\longrightarrow u^{(1)} \\
u_{h,\dt,\M^*} \underset{m\to \infty} \longrightarrow u^{(2)}
\end{cases}\quad 
\text{a.e. in $Q_T$}.
\ee
Setting $u= \left(u^{(1)}+u^{(2)}\right)/2$, we get that 
\be\label{eq:ae-conv}
u_{h,\dt}\underset{m\to \infty}\longrightarrow u.
\ee
Moreover, as the sequences $\left(u_{h,\dt}\right)_m$, 
$\left(u_{h,\dt, \M}\right)_m$, and $\left(u_{h,\dt, \M^*}\right)_m$ are uniformly equi-integrable in $L^p(0,T,L^1(\O))$ and  converge point-wise, thanks to  \eqref{eq:ae-conv12}--\eqref{eq:ae-conv}, we can apply Vitali's convergence theorem to claim that the sequences converge strongly in $L^p(0,T,L^1(\O))$.
\smallskip

\noindent {\em Step 2.} Here, the goal is to show that the sequences $\left(u_{h,\dt}\right)_{m\ge1}$, 
$\left(u_{h,\dt,\M}\right)_{m\ge1}$ and $\left(u_{h,\dt,\M^*}\right)_{m\ge1}$ share the same limit $u$, 
i.e., $u^{(1)} = u^{(2)} = u$. 
As a consequence of \eqref{dissipation_control}, one has 
$$
 \sum_{n=1}^{N_T} \dt\llbracket \mathcal P(g_\T^{n}), g_\T^{n}\rrbracket_\T \leq C, 
$$
hence
$$
\left\| \log(u_{h,\dt,\M}) + V_{h,\M} - \log(u_{h,\dt,\M^*}) - V_{h,\M^*}\right\|^2_{L^2(Q_T)} \leq C\, {\rm size(\T)}^\beta.
$$
The regularity of the exterior potential $V$ implies that 
$$
\left\|V_{h,\M} - V_{h,\M^*}\right\|^2_{L^2(Q_T)} \leq C\, {\rm size(\T)}^2, 
$$
so that one gets that 
$$
\left\| \log(u_{h,\dt,\M})  - \log(u_{h,\dt,\M^*}) \right\|^2_{L^2(Q_T)} \leq C\, {\rm size(\T)}^\beta.
$$
Up to a subsequence, this ensures that $\log(u_{h,\dt,\M})  - \log(u_{h,\dt,\M^*})$ tends to $0$ a.e. in $Q_T$.
But owing to \eqref{eq:ae-conv12}, $\log(u_{h,\dt,\M})$ tends to $\log(u^{(1)})$ while $\log(u_{h,\dt,\M^*})$ tends to $\log(u^{(2)})$.
Then $\log(u^{(1)})=\log(u^{(2)})$ a.e. in $Q_T$, thus $u^{(1)} = u^{(2)}=u$. 
\smallskip

\noindent {\em Step 3.} In order to conclude the proof of Proposition~\ref{prop:strong-conv}, 
it remains to check that \eqref{eq:convL1-strong-hD} holds. To this end, we compute $u_{h,\dt, \DD} - u_{h,\dt}$ 
on each quarter diamond to get 
$$\left\|u_{h,\dt, \DD} - u_{h,\dt}\right\|_{L^1(Q_T)} 
\leq \; \frac14 \sum_{n=1}^{N_T} \dt \sum_{\Ds \in \DD} \md \left(\left|u_K^n - u_L^n\right| + \left|u_{K^*}^n - u_{L^*}^n\right|\right).
$$
Using the identity $|a-b| = \left(\sqrt{a}+\sqrt{b}\right) \left|\sqrt{a}-\sqrt{b}\right|$ for $a,b \ge 0$ , 
one gets 
\begin{align*}
\left\|u_{h,\dt, \DD} - u_{h,\dt}\right\|_{L^1(Q_T)}
\leq & \, \frac1{4} \sum_{n=1}^{N_T} \dt \sum_{\Ds \in \DD} \md \left( \left(\xi_K^n + \xi_L^n\right)\left|\xi_K^n - \xi_L^n\right|
+ \left(\xi_{K^*}^n + \xi_{L^*}^n\right)\left|\xi_{K^*}^n - \xi_{L^*}^n\right| \right) \\
\leq & \,  \sum_{n=1}^{N_T} \dt \sum_{\Ds \in \DD} \md r^\Ds(\xi_\T^n) \left(\left|\xi_K^n - \xi_L^n\right| + \left|\xi_{K^*}^n - \xi_{L^*}^n\right| 
\right).
\end{align*}
Owing to Cauchy-Schwarz inequality, there holds 
$$
\left\|u_{h,\dt, \DD} - u_{h,\dt}\right\|_{L^1(Q_T)}  \leq  \left( \sum_{n=1}^{N_T} \dt \sum_{\Ds \in \DD} \md r^\Ds(u_\T^n) \right)^{1/2}
  \left( \sum_{n=1}^{N_T} \dt \sum_{\Ds \in \DD} \md \left|\delta^\Ds \xi_\T^n \right|^2  \right)^{1/2}.
$$
It is easy to verify that $B_{\sig}^\Ds \ge \frac1{\theta^\star}$ and $B_{\sige}^\Ds \ge\frac1{\theta^\star}$ (see \eqref{eq:bbB} for their definition). 
Hence, using \eqref{est:u_D_L1}, it provides 
$$
\left\|u_{h,\dt, \DD} - u_{h,\dt}\right\|_{L^1(Q_T)} \leq  C \; \size(\T)  \left( \sum_{n=1}^{N_T} \dt \sum_{\Ds \in \DD} \delta^\Ds \xi_\T^n 
\cdot \bbB^\Ds  \delta^\Ds \xi_\T^n \right)^{1/2}.
$$
Using~\eqref{eq:Fisher-01} together with~\eqref{eq:Fisher-sum}, we obtain that 
$$
\left\|u_{h,\dt, \DD} - u_{h,\dt}\right\|_{L^1(Q_T)} \leq  C \; \size(\T), 
$$
thus $u_{h,\dt,\DD}$ also converges towards $u$ in $L^1(Q_T)$. 
To get the convergence in $L^p(0,T,L^1(\O))$ for any $p\in [1,\infty)$, 
it only remains to write
\begin{multline*}
\left\|u_{h,\dt, \DD} - u_{h,\dt}\right\|_{L^p(0,T,L^1(\O))} 
\leq  \; \left\|u_{h,\dt, \DD} - u_{h,\dt}\right\|_{L^\infty(0,T,L^1(\O))}^{\frac{p-1}p}  \left\|u_{h,\dt, \DD} - u_{h,\dt}\right\|_{L^1(Q_T)}^{1/p} \\
\leq  \; \left(\left\|u_{h,\dt, \DD}\right\|_{L^\infty(0,T,L^1(\O))} + \left\|u_{h,\dt} \right\|_{L^\infty(0,T,L^1(\O))} \right)^{\frac{p-1}p}  
\left\|u_{h,\dt, \DD} - u_{h,\dt}\right\|_{L^1(Q_T)}^{1/p}
\end{multline*}
and to use~\eqref{eq:mass} and \eqref{est:u_D_L1}.
\end{proof}

The purpose of the following statement is the uniform in time weak-$L^1$ in space convergence of the approximate solution 
towards $u$.
\begin{prop}\label{prop:unif_L1-w}
Up to the extraction of an additional subsequence, 
\be\label{eq:unif_L1-w}
u_{h_m,\dt_m}(\cdot,t) \underset{m\to\infty}\longrightarrowÂ u(\cdot, t) \quad \text{in the $L^1(\O)$-weak sense for all $t\in [0,T]$}.
\ee
Moreover, the limit function $u$ satisfies 
$$\sup_{t\in[0,T]} \int_\O H(u) \d\x \leq C$$
for some $C$ depending only $u_0$ and $V$.
\end{prop}
\begin{proof}
Let $R>0$ be arbitrary. 
As a consequence of the de La Vall\'ee Poussin theorem~\cite{DM78} the space 
$$
E_R = \left\{ f: \O\to\R_+\; \middle|\; \int_\O f \d\x = \int_\O u_0\d\x \; \text{and}\; \int_\O H(f) \dx \leq R\right\}
$$
is equi-integrable in $L^1(\O)$, thus it follows from Dunford-Pettis theorem that $E_R$ is relatively compact for the weak-$L^1(\O)$ topology. Since the function $f \mapsto \int_\O H(f) \d\x$ is lower semi-continuous, any limit value for a sequence of $E_R$ 
also belongs to $E_R$, hence $E_R$ is closed, thus compact. 

Since $\O$ is bounded and because $E_R$ is equi-integrable, the $L^1(\O)$-weak topology coincides with the 
topology corresponding to the narrow convergence of measures restricted to $E_R$. It can thus be endowed with the 
bounded-Lipschitz metric:
$$
{\rm dist}_{\rm BL}(f_n, f) \underset{n\to\infty}\longrightarrow 0 \quad \text{iff} \quad f_n \underset{n\to\infty}\longrightarrow f \; \text{weakly in $L^1(\O)$}.
$$
In the above formula, $f$ and $f_n$ ($n\ge 1$) belong to $E_R$, and 
$$
{\rm dist}_{\rm BL}(f, g) = \sup_{
 \|\grad \varphi\|_\infty \leq 1
}\int_\O (f-g) \varphi\, \d\x.
$$
We refer for instance to \cite[Theorem 5.9]{Santambrogio_OTAM} for the equivalence of the topology induced 
by the bounded-Lipschitz distance with the one of narrow convergence of positive measures. The fact that this latter topology 
coincides with the weak-$L^1(\O)$ topology on $E_R$ results from its equi-integrability.

As a consequence of Lemma~\ref{lem:estsolapp}, there exists $R$ such that $u_{h_m,\dt_m}(\cdot,t)$ belongs to $E_R$ for all $t \in [0,T]$ 
and all $m \ge 1$.
Let $\tau \in (0,T)$, and let $t \in (0,T-\tau)$ and let $\varphi: \ov \O \to \R$ be Lipschitz continuous 
with $\|\grad \varphi\|_\infty \leq 1$. 
Define $\varphi_{\T_m}$ as in~\eqref{eq:proj_phi}, then (we remove the subscript $m$ for legibility)
$$
\int_\O \left(u_{h,\dt}(\x,t+\tau) -  u_{h,\dt}(\x, t)\right) \varphi(\x) \d \x = 
\left\llbracket u_{\T}^{N^{(2)}}  - u_{\T}^{N^{(1)}}, \varphi_{\T} \right\rrbracket_{\T}
$$
where $N^{(1)}$ and $N^{(2)}$ are the positive integers such that 
$$
\left(N^{(1)} - 1\right) \dt < t \leq N^{(1)} \dt, \qquad \left(N^{(2)} - 1\right) \dt < t + \tau\leq N^{(2)} \dt.
$$
Using the scheme~\eqref{sch_formcompacte}, we obtain that 
$$
\left\llbracket u_{\T}^{N^{(2)}}  - u_{\T}^{N^{(1)}}, \varphi_{\T} \right\rrbracket_{\T} = 
\sum_{n=N^{(1)} +1}^{N^{(2)}} \dt \left( 
T_{\DD} \left( u_\T^{n} ; g_\T^{n}, \varphi_\T \right) + \kappa \left\llbracket \Pp^\T g_\T^n, \varphi_\T \right\rrbracket_\T \right).
$$
Cauchy-Schwarz inequality yields 
\begin{multline*}
\left\llbracket u_{\T}^{N^{(2)}}  - u_{\T}^{N^{(1)}}, \varphi_{\T} \right\rrbracket_{\T} \leq 
\left( 
\sum_{n=N^{(1)} +1}^{N^{(2)}} \dt \left( 
T_{\DD} \left( u_\T^{n} ; g_\T^{n}, g_\T^{n} \right) + \kappa \left\llbracket \Pp^\T g_\T^n, g_\T^{n} \right\rrbracket_\T \right)
\right)^{1/2}\\
\times \left( 
\sum_{n=N^{(1)} +1}^{N^{(2)}} \dt \left( 
T_{\DD} \left( u_\T^{n} ; \varphi_\T, \varphi_\T \right) + \kappa \left\llbracket \Pp^\T \varphi_\T, \varphi_\T \right\rrbracket_\T \right)
\right)^{1/2}.
\end{multline*}
The first term in the right-hand side is uniformly bounded by $\bbE(0)$ thanks to \eqref{dissipation_control}.
On the other hand, Estimate \eqref{eq:grad-Linf0} together with $\|\grad \varphi \|_\infty \leq 1$ provide that 
$$
T_{\DD} \left( u_\T^{n} ; \varphi_\T, \varphi_\T \right) \leq C \int_\O u_{h,\DD}^n \d\x, \qquad \forall n \in \{1,\dots, N_T\}.
$$
Owing to \eqref{est:u_D_L1}, we get that 
$$
T_{\DD} \left( u_\T^{n} ; \varphi_\T, \varphi_\T \right) \leq C, \qquad \forall n \in \{1,\dots, N_T\}.
$$
Hereby, we deduce that 
$$
\left\llbracket u_{\T}^{N^{(2)}}  - u_{\T}^{N^{(1)}}, \varphi_{\T} \right\rrbracket_{\T} \leq  C \sqrt{(N^{(2)} - N^{(1)} ) \dt} 
\leq C \sqrt{\tau +  \dt}, 
$$
and since $\varphi$ was chosen arbitrarily in $\{\varphi: \O \to \R\, | \, \|\grad \varphi\|_\infty \leq 1\}$, we obtain that 
$$
{\rm dist}_{\rm BL}(u_{h,\dt}(\cdot,t+\tau), u_{h,\dt}(\cdot, t)) \leq C\sqrt{\tau + \dt}, \qquad \forall t \in [0,T-\tau).
$$
We can apply the refined version of Arzel\`a-Ascoli theorem \cite[Proposition 3.3.1]{AGS08} (see also 
\cite[Theorem 4.26]{Kangourou}) and claim that $u_{h,\dt}$ converges uniformly towards $u \in C([0,T);E_R)$, $E_R$ being endowed 
with the $L^1(\O)$-weak topology.
\end{proof}

\subsection{Identification of the limit}

The goal of this section is to show that the limit function $u$ exhibited 
in Proposition~\ref{prop:strong-conv} is a weak solution in the sense 
of Definition~\ref{Def:weak}. 

The second and last point to check to complete the proof of  Theorem~\ref{thm:main2} is 
the fact that $u$ is a solution to~\eqref{pb} in the distributional sense, i.e., that the 
weak formulation~\eqref{eq:weak} is fulfilled. This is the purpose of the following statement. 

\begin{prop}\label{prop:identify}
Let $u$ be as in Proposition~\ref{prop:strong-conv}, then $u$ satisfies the weak formulation~\eqref{eq:weak}.
\end{prop}
\begin{proof}
Here again, we remove the subscript $m$ when it appears us to be detrimental for the readability. 
Let $\varphi \in C^\infty_c(\ov \O \times [0,T))$, then denote by 
$$
\varphi_K^n = \varphi(\x_K, t^n), \qquad \varphi_{K^*}^n = \varphi(\x_{K^*}, t^n), \qquad 
\forall K \in \ov \M, \; \forall K^* \in \ov{\M^*}.
$$
Choosing the corresponding $\psi_\T = \varphi_\T^n$ in~\eqref{sch_formcompacte}, multiplying by $\dt$ and 
summing over $n \in \{0,\dots, N_{T-1}\}$ leads to 
\be\label{eq:ABC_m}
A_m + B_m + C_m = 0, 
\ee
where
\begin{align*}
A_m =&\, \sum_{n=0}^{N_{T-1}} \llbracket u_\T^{n+1} - u_\T^n, \varphi_\T^n \rrbracket_\T, \\
B_m = &\, \sum_{n=0}^{N_{T-1}} \dt \, T_\DD (u_\T^{n+1}; g_\T^{n+1}, \varphi_\T^n ), \\
C_m = &\, \kappa \sum_{n=0}^{N_{T-1}} \dt\, \llbracket \mathcal P^\T g_\T^{n+1} , \varphi_\T^n \rrbracket_\T.
\end{align*} 
For the terms $A_m$ and $C_m$, we can proceed as is~\cite{CKM15} to get 
\be\label{eq:ACm}
A_m \underset{m\to\infty}\longrightarrow -\iint_{Q_T} u \p_t \varphi \, \d\x \d t - \int_\O u_0 \varphi(\cdot, 0) \d\x,
\qquad \text{and}\qquad
C_m \underset{m\to\infty}\longrightarrow 0.
\ee
Let us now detail the treatment of the term $B_m$ and start by splitting it into
\be\label{eq:Bm}
B_m = B_{1,m} + B_{2,m}, 
\ee
where 
\begin{align*}
B_{1,m} = &\, \sum_{n=0}^{N_{T-1}} \dt \sum_{\Ds \in \DD} r^\Ds(u_\T^{n+1}) \delta^\Ds V_\T \cdot \bbA^\Ds \delta^\Ds \varphi_\T^n \\
=&\, \iint_{Q_T} u_{h,\dt,\DD}(\x,t) \bLambda_{h,\DD} \grad^h V_h(\x) \cdot  \grad^h \varphi_{h,\dt}(\x, t-\dt) \d\x \d t,
\end{align*}
and
$$
B_{2,m} = \sum_{n=0}^{N_{T-1}} \dt \sum_{\Ds \in \DD} r^\Ds(u_\T^{n+1}) \delta^\Ds \log(u_\T^{n+1}) \cdot \bbA^\Ds \delta^\Ds \varphi_\T^n.
$$
Since $V$ and $\varphi$ are smooth functions, one has 
$$
\grad^h V_h \underset{m\to\infty} \longrightarrow \grad V \quad \text{uniformly on}\; \O, 
\qquad
\grad^h  \varphi_{h,\dt}(\cdot, \cdot-\dt)  \underset{m\to\infty} \longrightarrow \grad \varphi \quad \text{uniformly on}\; Q_T, 
$$
whereas $\bLambda_{h,\DD}$ converges a.e. towards $\bLambda$.
Then it follows from \eqref{eq:convL1-strong-hD} that 
\be\label{eq:B1m} 
B_{1,m} \underset{m\to\infty} \longrightarrow  \iint_{Q_T} u \bLambda \grad V \cdot \grad \varphi \d\x \d t.
\ee
The last term $B_{2,m}$ is treated following the method proposed in~\cite{CG2016} that consists in writing 
$$
B_{2,m} = B_{2,m}^{(1)}+ B_{2,m}^{(2)}, 
$$
with 
\begin{align*}
B_{2,m}^{(1)} =& \, 2 \iint_{Q_T} \sqrt{u_{h,\dt,\DD}} \bLambda_{h,\DD} \grad^h \xi_{h,\dt} \cdot \grad^h \varphi_{h,\dt}(\cdot, \cdot - \dt) \d\x \d t, \\
B_{2,m}^{(2)} =& \, \sum_{n=1}^{N_T} \dt \sum_{\Ds \in \DD} \sqrt{r^\Ds(u_\T^n)}
\delta^\Ds \log(u_\T^n) \cdot \\
& \hspace{3cm}\begin{pmatrix} {\xi_{KL}^n} - \sqrt{r^\Ds(u_\T^n)} & 0 \\
0 &  {\xi_{K^*L^*}^n} - \sqrt{r^\Ds(u_\T^n)} 
\end{pmatrix} 
\bbA^\Ds \delta^\Ds \varphi_\T^{n-1},
\end{align*}
where we have set 
\be\label{eq:uKL}
{\xi_{KL}^n} =  \begin{cases}
2 \frac{{\xi_K^n} - {\xi_L^n}}{\log(u_K^n) - \log(u_L^n)} & \text{if}\; u_K^n \neq u_L^n, \\
{\xi_{K}^n}  &\text{if}\; u_K^n = u_L^n,
\end{cases}
\qquad
{\xi_{K^*L^*}^n} =  \begin{cases}
2 \frac{{\xi_{K^*}^n} - \xi_{L^*}^n}{\log(u_{K^*}^n) - \log(u_{L^*}^n)} & \text{if}\; u_{K^*}^n \neq u_{L^*}^n, \\
{\xi_{K^*}^n}  &\text{if}\; u_{K^*}^n = u_{L^*}^n.
\end{cases}
\ee
We know that, up to a subsequence, $\sqrt{u_{h,\dt,\DD}}$ converges strongly in $L^2(Q_T)$ towards $\sqrt{u}$, 
whereas $\grad^h \xi_{h,\dt}$ converges weakly in $L^2(Q_T)^2$ towards $\grad \sqrt{u}$, and 
$\grad^h \varphi_{h,\dt}(\cdot, \cdot - \dt)$ converges uniformly towards $\grad \varphi$. Thus we can pass to the 
limit in $B_{2,m}^{(1)}$ and obtain that 
\be\label{eq:B_2m^1}
B_{2,m}^{(1)} \underset{m\to\infty}\longrightarrow  2 \iint_{Q_T} \sqrt{u} \bLambda \grad \sqrt{u} \cdot \grad \varphi \d\x \d t 
= \iint_{Q_T} \bLambda \grad u \cdot \grad \varphi \d\x \d t.
\ee

In order to show that $B_{2,m}^{(2)}$ tends to $0$, we need a few preliminaries. 
Owing to the definition~\eqref{eq:uKL} of $\xi_{KL}^n$ and $\xi_{K^*L^*}^n$, one always has 
$$
\min\left(\xi_K^n, \xi_L^n\right) \leq \xi_{KL}^n \leq \max\left(\xi_K^n, \xi_L^n\right), 
\qquad 
\min\left(\xi_{K^*}^n, \xi_{L^*}^n\right) \leq \xi_{K^*L^*}^n \leq \max\left(\xi_{K^*}^n, \xi_{L^*}^n\right), 
$$
so that, denoting by $\widetilde \xi_{h,\dt,\DD}$ and $\widetilde \xi^*_{h,\dt,\DD}$ the functions defined almost everywhere by 
$$
\widetilde \xi_{h,\dt,\DD}(\x,t) = {\xi_{KL}^n} \quad \text{and}\quad  \widetilde \xi^*_{h,\dt,\DD}(\x,t) ={\xi_{K^*L^*}^n}
\quad \text{if}\; (\x,t) \in \Ds \times (t_{n-1}, t_n],
$$
one obtains that 
$$
\| \xi_{h,\dt,\M} - {\widetilde \xi_{h,\dt,\DD}} \|_{L^2(Q_T)}^2 
\leq  \sum_{n=1}^{N_T} \dt \sum_{\Ds \in \DD} \md |\xi_K^n - \xi_L^n|^2 \leq C \size(\T)^2  \|\grad^h \xi_{h,\dt}\|_{L^2(Q_T)}^2
$$
for some $C$ depending only on $\theta^\star$ and $\bLambda$.
Similarly, one has 
$$
\| \xi_{h,\dt,\M^*} - {\widetilde \xi^*_{h,\dt,\DD}} \|_{L^2(Q_T)} \leq C \size(\T) \|\grad^h \xi_{h,\dt}\|_{L^2(Q_T)}.
$$
Bearing in mind that $\grad^h \xi_{h,\dt}$ is uniformly bounded in $L^2(Q_T)^2$ w.r.t. $m$, this ensures in particular that 
\be\label{eq:conv-tilde-xi}
 {\widetilde \xi_{h,\dt,\DD}} \underset{m\to\infty}\longrightarrow \xi = \sqrt{u} \quad \text{and} \quad 
 {\widetilde \xi^*_{h,\dt,\DD}} \underset{m\to\infty}\longrightarrow \xi = \sqrt{u}  \quad \text{in}\; L^2(Q_T).
\ee
As a consequence of~\eqref{eq:convL1-strong-hD}, the function $\sqrt{u_{h,\dt,\DD}}$ also converges towards $\xi$ in $L^2(Q_T)$. 

We now have at hand all the necessary material to study $B_{2,m}^{(2)}$. It results from Cauchy-Schwarz inequality that 
\begin{multline*}
B_{2,m}^{(2)} \leq \left(\sum_{n=1}^{N_T} \dt \widehat \bbI_\T^n \right)^{1/2} \\
\times \left(\sum_{n=1}^{N_T} \dt \sum_{\Ds \in \DD} 
\begin{pmatrix} \left|{\xi_{KL}^n} - \sqrt{r^\Ds(u_\T^n)}\right|^2 & 0 \\
0 &  \left|{\xi_{K^*L^*}^n} - \sqrt{r^\Ds(u_\T^n)}\right|^2 
\end{pmatrix} \delta^\Ds \varphi_\T^{n-1} \cdot \bbB^\Ds  \delta^\Ds \varphi_\T^{n-1} 
\right)^{1/2}.
\end{multline*}
Thanks to the regularity of the mesh and of $\varphi$, one has 
$$
\delta^\Ds \varphi_\T^{n-1} \cdot \bbB^\Ds  \delta^\Ds \varphi_\T^{n-1} \leq C \md
$$
for some $C>0$ depending only on $\theta^\star$ and on $\|\grad \varphi\|_\infty$, whereas Lemma~\ref{lem:Fisher-0} and \eqref{dissipation_control} ensure 
that $$\sum_{n=1}^{N_T} \dt \widehat \bbI_\T^n \leq C.$$
Hence, we get 
$$
B_{2,m}^{(2)} \leq C \left(\|\widetilde \xi_{h,\dt,\DD} - \sqrt{u_{h,\dt,\DD}}\|_{L^2(Q_T)} + 
\|\widetilde \xi^*_{h,\dt,\DD} - \sqrt{u_{h,\dt,\DD}}\|_{L^2(Q_T)}\right).
$$
Using \eqref{eq:conv-tilde-xi} together with the fact that $\sqrt{u_{h,\dt,\DD}}$ also converges towards $\xi$ in $L^2(Q_T)$, 
we get that 
\be\label{eq:B_2m^2}
B_{2,m}^{(1)} \underset{m\to\infty}\longrightarrow  0.
\ee
We conclude the proof by putting together the statements~\eqref{eq:ACm}, \eqref{eq:Bm}, \eqref{eq:B1m}, \eqref{eq:B_2m^1} and 
\eqref{eq:B_2m^2}  in~\eqref{eq:ABC_m}.
\end{proof}

\section{Numerical experiments}\label{sec:num}

\subsection{About the practical implementation}

The nonlinear system \eqref{scheme} is solved thanks to Newton's method. In order to avoid 
the singularity of the $\log$ near $0$, the sequence ${(u_\T^{n+1,i})}_{i\ge 0}$ to compute $u_\T^{n+1}$ 
from the previous state ${(u_\T^{n})}_{i\ge 0}$ is initialized by $u_\T^{n+1,0} = \max (u_\T^n, 10^{-12})$. 
In practice, we observe that the threshold criterion is not used.
As a stopping criterion, we require the $\ell^1$-norm of the residual to be smaller than $10^{-10}$.

\subsection{Convergence w.r.t. to the discretization parameters}

We test our method on a test case inspired from the one in~\cite{CG2016}. We set  $\O = (0,1)^2$, and $V(x_1,x_2) = -x_2$. 
The exact solution $u_{\rm ex}$ is then defined by 
$$
u_{\rm ex}((x_1, x_2), t) =  e^{-\alpha t+ \frac{x_2}{2}}\left(\pi\cos(\pi x_2)+\frac12\sin(\pi x_2)\right)+\pi e^{\left(x_2-\frac12\right)}
$$
with $\alpha = \pi^2 + \frac14$.
We choose $u_0 = u_{\rm ex}(\cdot, 0)$. Note that $u_0$ vanishes on $\{x_2 = 1\}$. 

In order to illustrate the convergence and the robustness of our method, we study its convergence on two sequences 
of meshes. The first sequence of primal meshes is made of successively refined Kershaw meshes. 
The second sequence of primal meshes is the so-called quadrangle meshes {\tt mesh\_quad\_i} of the FVCA8 benchmark on incompressible flows. 
One mesh of each sequence is depicted in Figure~\ref{fig:mesh}.
In the refinement procedure, the time step is divided by $4$ when the mesh size is divided by $2$.

\begin{figure}[htb]
\begin{center}
\includegraphics[height=3.3cm]{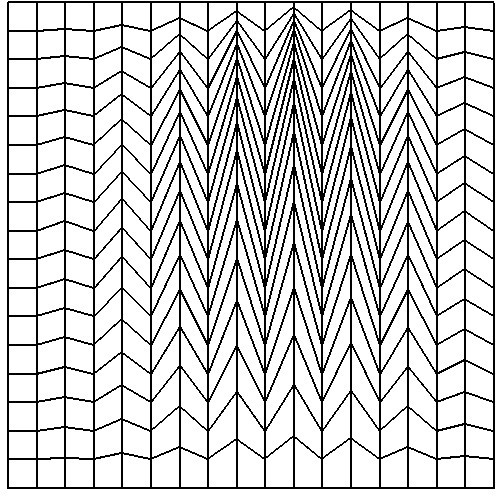}
\hspace{1cm}
\includegraphics[height=3.3cm]{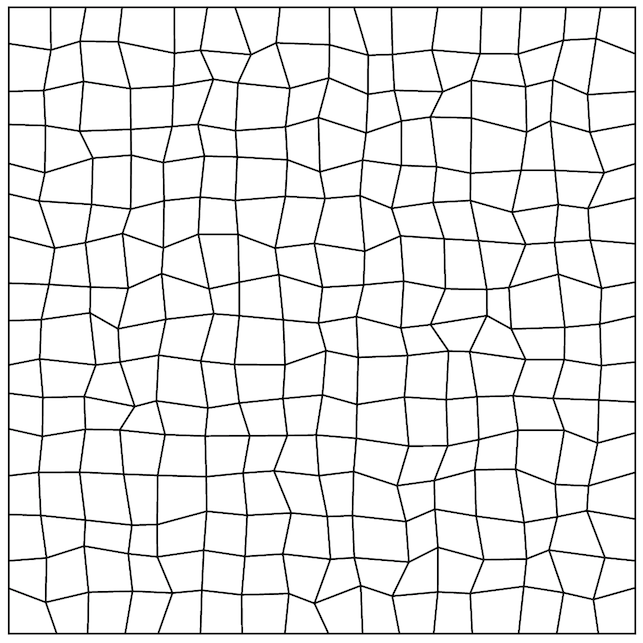}

\end{center}
\caption{Left: First Kershaw mesh. Right: Third quadrangle mesh. 
}
\label{fig:mesh}
\end{figure}

We have introduced a penalization operator in order to  prove that reconstruction  on the primal mesh $u_{h,\M,\dt} $ and the reconstruction on the dual mesh  $u_{h,\Mie,\dt}$ converge to the same limit. In Table~\ref{tab:uMuMie}, we compute normU the $L^2(\O\times(0,T))^2$  norm
of the difference between the two different reconstructions and ordU the corresponding convergence order for different values of $\kappa$ the penalization parameter.
We numerically observe the same result : the two reconstructions 
converge to the same limit even if $\kappa$ is zero.

\begin{table}[!ht]
 \begin{center}
  \begin{tabular}{|c|c||c|c||c|c|}
  \hline
    M&dt&\multicolumn{2}{|c|}{$\kappa=0$}&\multicolumn{2}{|c|}{$\kappa=10^{-1}$}\\ \hline
    & &normU &ordU&normU &ordU\\ \hline
1&4.032E-03&   1.798E-01&---    &  1.796E-01&---  \\ \hline
2&1.008E-03&   9.316E-02&        0.95&9.313E-02&        0.94\\ \hline
3&2.520E-04&   4.717E-02&       1.03& 4.716E-02&       1.03\\ \hline
4&6.300E-05&   2.361E-02&        1.05& 2.361E-02&        1.05\\ \hline
5&1.575E-05&   1.135E-02&        0.87&1.135E-02&        0.87\\ \hline
  \end{tabular} 
 \end{center}  
 \caption{Numerical results on the Quadrangle mesh family, final time T=0.25.}
 \label{tab:uMuMie}
\end{table}

In the following of this section, the penalization parameter $\kappa$ is set to zero.
In Tables~\ref{tab:Ker} and \ref{tab:quad},  the quantities erru and errgu respectively denote the 
$L^\infty((0,T);L^2(\O))$ error on the solution and the 
$L^2(\O\times(0,T))^2$ error on the gradient, whereas ordu and ordgu are the corresponding convergence orders. 
It appears that the method is slightly more than second order accurate w.r.t. space. 

The maximal (resp. mean) number of Newton iterations 
by time step is denoted by $N_{\text{max}}$ (resp. $N_{\text{mean}}$). 
We observe that the needed number of Newton iterations starts from a reasonably small value and falls 
down to $1$ after a small number of time steps. Therefore, our method does not imply an important 
extra computational cost when compared to linear methods. 
Eventually, we can check that the minimal value min $u_\T^n$ remains 
strictly greater than 0, as proved in Lemma~\ref{lem:pos}.

\begin{table}[!ht]
 \begin{center}
  \begin{tabular}{|c|c||c|c|c|c|c|c|c|}
  \hline
    M&dt&errgu&ordgu&erru&ordu&$N_{\text{max}}$&$N_{\text{mean}}$&Min $u^n$\\ \hline
1&2.0E-03&  6.693E-02& --- &7.254E-03&--- &9&2.15 &1.010E-01\\\hline
2&5.0E-04&  2.353E-02&1.54 &1.751E-03&2.09&8&2.02 &2.582E-02\\\hline
3&1.25E-04& 1.235E-02&1.61 &7.237E-04&2.20&7&1.49 &6.488E-03\\\hline
4&3.125E-05&7.819E-03&1.60 &3.962E-04&2.11&7&1.07 &1.628E-03 \\\hline
5&3.125E-05&5.507E-03&1.58 &2.556E-04&1.98&7&1.04 &1.628E-03\\\hline
  \end{tabular} 
 \end{center}  
 \caption{Numerical results on the Kershaw mesh family, final time T=0.25.}
 \label{tab:Ker}
\end{table}
\begin{table}[!ht]
 \begin{center}
  \begin{tabular}{|c|c||c|c|c|c|c|c|c|}
  \hline
    M&dt&errgu&ordgu&erru&ordu&$N_{\text{max}}$&$N_{\text{mean}}$&Min $u^n$\\ \hline
1&4.032E-03&1.754E-01 &--- &2.149E-02 &--- &9& 2.26& 1.803E-01\\ \hline
2&1.008E-03&5.933E-02 &1.56&5.055E-03 &2.08&9& 2.04& 5.079E-02\\ \hline
3&2.520E-04&2.294E-02 &1.44&1.299E-03 &2.06&8& 1.96& 1.352E-02\\ \hline
4&6.300E-05&8.631E-03 &1.48&3.256E-04 &2.09&8& 1.22& 3.349E-03\\ \hline
5&1.250E-05&2.715E-03 &1.37&7.702E-05 &1.70&7& 1.01& 8.695E-04 \\\hline
  \end{tabular} 
 \end{center}  
 \caption{Numerical results on the Quadrangle mesh family, final time T=0.25.}
  \label{tab:quad}
\end{table}

\subsection{Long time behavior}

In this section, the penalisation parameter $\kappa$ is set to zero.
The discrete stationary solution $u_\T^\infty$ is defined by 
$u_K^\infty = \rho e^{-V(x_K)}$ and $u_{K^\ast}^\infty = \rho^\ast e^{-V(x_{K^\ast})}$ for $K \in \overline {\mathfrak{M}}$ 
and $K^\ast \in \overline {\mathfrak{M}}^\ast$, the quantities
$\rho$ and $\rho^\ast$ being fixed so that  
$
\sum_{K\in \M} u_K^\infty m_K= \sum_{K\in \overline \M^\ast} u_{K^\ast}^\infty m_{K^\ast}= \int_\O u_0(x) dx.
$
In order to give an evidence of the good large-time behavior of our scheme, we plot in Figure~\ref{fig:EntRel} 
the evolution of the relative energy 
$$
{\mathbb E}_\T^{n}-{\mathbb E}_\T^{\infty} = \left\llbracket u_\T^n \log\left(\frac{u_\T^n}{u_\T^\infty}\right) - u_\T^n + u_\T^\infty, 1_\T\right\rrbracket_\T, \quad n \ge 0
$$
computed on the Kershaw meshes and on the quadrangle meshes. 
We observe the exponential decay of the relative energy, recovering on general grids the 
behavior of the Scharfetter-Gummel scheme~\cite{Chatard2011}.
\begin{figure}[htb]
\begin{center}
\includegraphics[height=5.3cm]{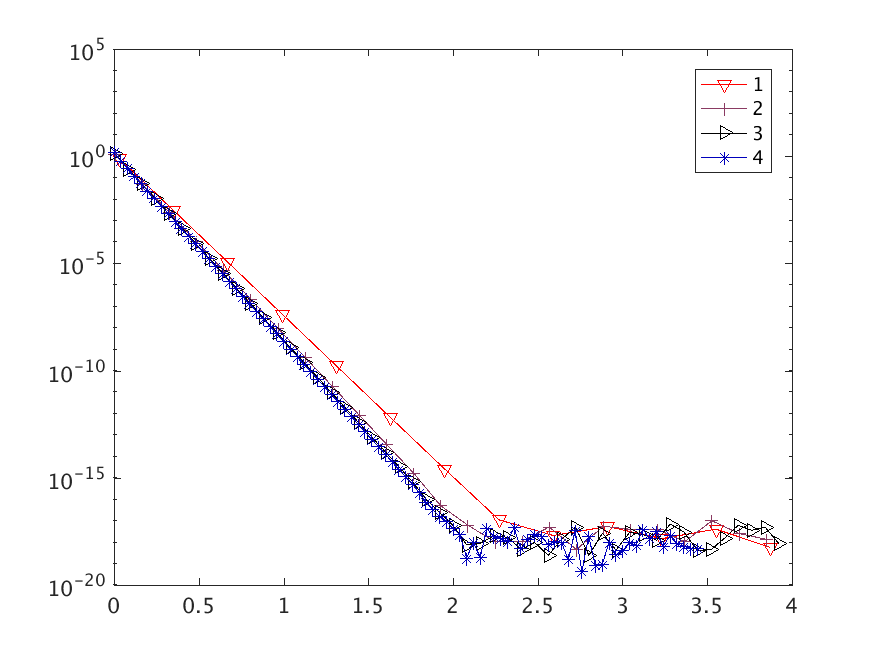}
\hspace{0.2cm}
\includegraphics[height=5.3cm]{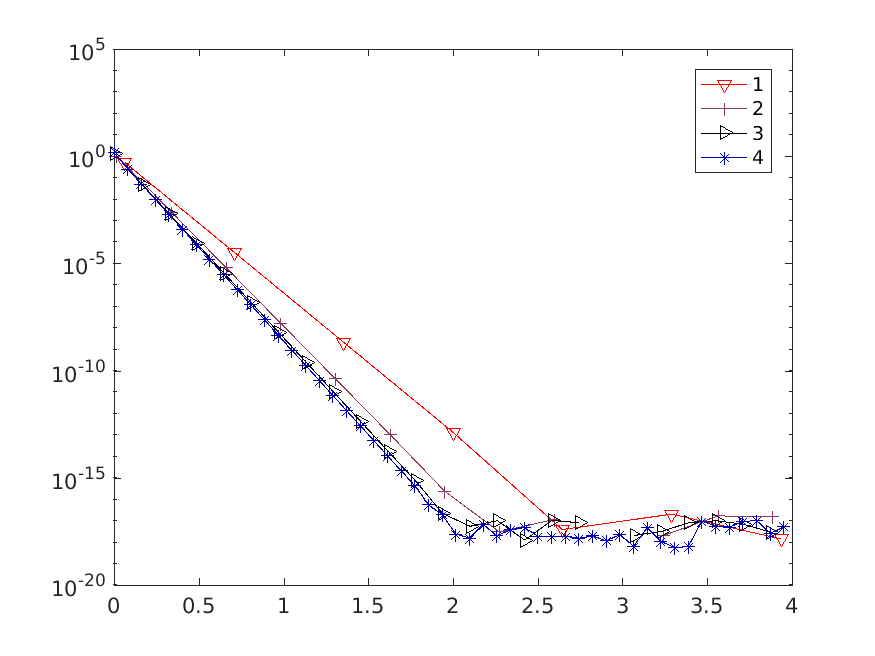}
\end{center}
\caption{Discrete relative energy ${\mathbb E}_\T^{n}-{\mathbb E}_\T^{\infty}$
 as a function of $n\Delta t$ computed on the  first four Kershaw meshes (on the left) and on the  first four quadrangle meshes (on the right).
}
\label{fig:EntRel}
\end{figure}

\appendix

\section{A trace inequality}

First, to a given vector $\ut=\left(\left(\uk\right)_{\petitk\in{\overline \M}},\left(\uke\right)_{\petitke\in{\overline \Mie}}\right) \in\Rt$
 defined on a DDFV mesh $\T$, we associate its {\em primal trace} $\gamma_{\p\M} \ut $ on $\p\O$ defined by 
\begin{equation*}\label{defuhbord}
 \gamma_{\p\M} \ut(\x)=\sumpriext\ul\boldsymbol{1}_{\l}(\x), \qquad \forall \x \in \p\O.
\end{equation*}
\begin{thm}[Trace inequality]\label{thmtracebis}
Let $\O$ be a convex polygonal domain of  $\R^2$ and $\T$ a DDFV mesh of this domain.
There exist $C>0$, depending only on $\O$ and $\theta^\star$, such that $\forall\ \ut \in\Rt $:
\begin{equation}\label{traceine}
 \|\gamma_{\p\M} \ut\|_{2,{\dO}}\leq C 
\left(|\ut|_{2,\T}+ \left\|\nabla^h u_h\right\|_{2}\right).
\end{equation}
\end{thm}
\begin{proof}
 The calculations are similar to those followed in \cite[Lemma  10.5]{EGH00} and in \cite[Theorem 7.1]{CKM15} for $L^1$-norm.
 The difference comes from the fact that here we define $\udM $ using the boundary primal mesh instead of the interior primal mesh.
 Adapting the proof of \cite[Theorem 7.1]{CKM15} in the $L^2$-norm, we get for $\k\in\M$ such that $\overline{\k}\cap\dO\ne0$ the inequality
 \[
  \sumdiamext\msig|\uk|^2\leq C 
\left(|\ut|_{2,\T}^2+ \left\|\nabla^h u_h\right\|_{2}^2\right).
 \]
 It implies
 \[
 \begin{aligned}
     \|\gamma_{\p\M} \ut\|_{2,{\dO}}^2&=\sumdiamext\msig |\ul-\uk+\uk|^2\\
     &\leq 2\sumdiamext\msig |\ul-\uk|^2+2 \sumdiamext\msig|\uk|^2\\
     &\leq 2\sumdiamext\msig |\ul-\uk|^2+ C 
\left(|\ut|_{2,\T}^2+ \left\|\nabla^h u_h\right\|_{2}^2\right).
 \end{aligned}
 \]
 Using the fact that $ \ul-\uk=\msige(\gradD\ut)\cdot\tkl$, we conclude 
 \[
  \|\gamma_{\p\M} \ut\|_{2,{\dO}}^2\leq 
C\left(|\ut|_{2,\T}^2+ (1+{\rm size}(\T))\left\|\nabla^h u_h\right\|_{2}^2\right).
 \]
\end{proof}

\bibliographystyle{plain}
\bibliography{biblio_CCK}

\end{document}